\title{\LARGE\textbf{
Families of degenerating Poincaré-Einstein metrics on $\mathbb{R}^4$}}
\author{Carlos A. \textsc{Alvarado
\quad \quad 
Tristan \textsc{Ozuch}
\quad \quad Daniel A. \textsc{Santiago}}
\vspace{0.2cm}\\
    \textsc{\footnotesize MIT Department of Mathematics\vspace{-0.1cm} Cambridge, MA} \vspace{-0.1cm}
}
\date{}

\documentclass[a4paper,10pt,reqno]{article}

\usepackage[colorinlistoftodos,prependcaption]{todonotes}

\usepackage{lmodern}

\usepackage{amsmath}
\usepackage{amsthm,xcolor}
\usepackage[colorlinks, linkcolor=red]{hyperref}
\usepackage[top=1in, lmargin=0.8in, rmargin=0.7in, marginparwidth=0.8in]{geometry}
\usepackage[english]{babel}
\usepackage{amssymb}
\usepackage{mathrsfs}
\usepackage{mathtools}
\usepackage{bbm}
\usepackage{amsthm}
\usepackage{fancyhdr}
\usepackage{tikz-cd}
\usepackage{comment}
\usepackage{etoolbox}
\usepackage{breqn}
\usepackage{esint}
\usepackage{appendix}
\hypersetup{
    colorlinks=true,
    linkcolor=blue,
    filecolor=magenta,      
    urlcolor=cyan,
	citecolor=blue
}
\usepackage{tikz}
\usepackage{tkz-tab}
\usepackage[center]{caption}
\usepackage{latexsym}
\usepackage{subcaption}
\usepackage{ucs} 

\usepackage[T1]{fontenc}
\usepackage[center]{titlesec}
\usepackage{xcolor}
\usepackage[bottom]{footmisc}
\usepackage{indentfirst}
\usepackage{xpatch}
\usepackage{chngcntr}
\usepackage{pgfplots}
\usepgfplotslibrary{fillbetween}
\usepackage{float}

\usepackage[colorinlistoftodos,prependcaption]{todonotes}

\pgfplotsset{width=10cm,compat=1.11}

\newcommand{\Addresses}{{% additional braces for segregating \footnotesize
		\bigskip
		\footnotesize
		
	 \textsc{MIT, Dept. of Math., 182 Memorial drive, Cambridge, MA 02139.}\par\nopagebreak
		\textit{E-mail address}, \texttt{ozuch@mit.edu}
	}}

\newtheorem{thm}{Theorem}[section]

\newtheorem{lem}[thm]{Lemma}
\newtheorem{prop}[thm]{Proposition}

\newtheorem{quest}[thm]{Question}

\newtheorem{rem}[thm]{Remark}

\DeclareMathOperator{\Rm}{\operatorname{Rm}}

\DeclareMathOperator{\Ric}{\operatorname{Ric}}

\date{}
\begin{document}

\maketitle 

\begin{abstract}
    We provide the first example of continuous families of Poincaré-Einstein metrics developing cusps on the trivial topology $\mathbb{R}^4$. We also exhibit families of metrics with unexpected degenerations in their conformal infinity only. These are obtained from the Riemannian version of an ansatz of Debever and Pleba\'nski-Demia\'nski. We additionally indicate how to construct similar examples on more complicated topologies.
\end{abstract}

\section*{Introduction}

An \textit{Einstein metric} satisfies for some real number $\Lambda$:
\begin{equation}
    \Ric(g)=\Lambda g.\label{Einstein condition}
\end{equation}
This is a central equation in Geometry and in several instances of Physics, especially in dimension $4$. A \textit{Poincaré-Einstein metric} is a noncompact Einstein metric with a specific asymptotic behavior giving rise to a \textit{conformal boundary metric} at infinity, the simplest example being the Poincaré model for hyperbolic space whose conformal infinity is the round sphere. Poincaré-Einstein metrics were first notably used to construct a number of conformal invariants of the boundary geometry; see \cite{MR837196,MR2858236}. More recently, they have also played an important role in the physics literature in relationship with AdS/CFT correspondence; see \cite{MR1633012,Biquard:2005yq}.

From several perspectives, dimension $4$ is a threshold dimension in topology and geometry. In this dimension, there are three ways for compact Einstein or Poincaré-Einstein metrics on a given manifold to degenerate: orbifold singularity formation, collapsing and cusp formation.

Orbifold formation has been widely studied and is now reasonably understood. Numerous examples of \emph{curves} of such degenerations have been produced in the Kähler and  Poincaré-Einstein settings, see \cite{MR1289837,MR3032330,MR3489703}. \textit{All} such degenerations have moreover been reconstructed by gluing-perturbation \cite{ozu1,ozu2}. 

Despite deep general results such as \cite{MR2188134}, the collapsing and cusp formation remain comparatively mysterious. The collapsing situation has received a lot of attention and many examples of \emph{curves} of Einstein metrics collapsing have been produced on K3 surfaces, see for instance \cite{MR4322391,MR3948228}. The third situation of cusp formation has however never been observed except from ``trivial'' examples of (warped) products of degenerating surfaces and from sequences of metrics requiring infinitely many different topologies \cite{MR2225518,MR2911887}. More concretely, the following question was left open:
\begin{quest}[\cite{MR2160865}]\label{question curves of metrics forming cusp}
    ``Another interesting open question is whether cusps can actually form within a given or fixed component of [the moduli space of Poincaré-Einstein metrics], on a fixed manifold $M$.''
\end{quest}
A simple but not so appealing example showing that this exists is the so-called \textit{topological black hole} metric. The metric is $V(r)^{-1}dr^2 + V(r) d\theta^2 +r^2 g_N$ for $V(r):= -1+r^2 -2m/r^2$ with $m$ large enough and $g_N$ the metric of a hyperbolic surface. Letting $g_N$ degenerate creates a cusp that extends to the conformal infinity. This naive example answers Anderson's question but, to the authors' knowledge, does not seem to have been mentioned before. This is still a $2$-dimensional behavior and we provide many more interesting examples here.

Another intriguing question is whether cusp formation requires some topology -- like orbifold degeneration requires nontrivial $2$-homology. Anderson conjectured that it was the case:
\begin{quest}[\cite{MR2160865}]\label{question cusp on R4}
    ``It would also be very interesting to know if the possible formation of cusps is restricted by the topology of the ambient manifold $M$. [...] One might conjecture for instance that on the $4$-ball cusp formation is not possible.''
\end{quest}
We instead provide explicit examples of continuous families of smooth Poincaré-Einstein metrics on $\mathbb{R}^4$ developing different kinds of cusps. We moreover find curves of metrics \textit{without} any degeneration in the bulk but forming various \textit{conical}, \textit{cusp} or \textit{naked} singularities in their conformal infinity.

\subsection*{Debever and Pleba\'nski-Demia\'nski's local family of metrics}

In this article, we study families of Poincaré-Einstein  metrics exhibiting the above three types of degenerations focusing on the least understood case of cusp formation. These examples are surprisingly \textit{explicitly given in coordinates} and are found in the families of Einstein metrics whose Lorentzian counterparts were discovered by Debever \cite{debever1971type} and which were given in more convenient coordinates by Pleba\'nski-Demia\'nski \cite{PLEBANSKI197698}. These metrics are known in the physics literature as Pleba\'nski-Demia\'nski metrics (PD metrics). PD metrics are algebraically special of Petrov \textit{type D} meaning (in the Riemannian setting) that at every point the selfdual and anti-selfdual parts of the Weyl curvature have repeated eigenvalues. This also equivalent to the \textit{ambiKähler} condition of \cite{MR3574879}: the metric is \textit{conformally Kähler} or \textit{Hermitian} in both orientations. This curvature condition forces toric symmetry by \cite{Goldblatt:1994iw}. 

The metrics of the PD family have a remarkably compact form \eqref{family PD with a}
and depend solely on two related \textit{quartic polynomials} $P$ and $Q$ of one variable. Still, despite their simplicity and their discovery in the early 70's, these explicit metrics, once extended to the Riemannian setting contain in some limits most known examples of Einstein metrics ($\mathbb{S}^4$, $\mathbb{S}^2\times \mathbb{S}^2$, Fubini-Study, Page's metric, Taub-NUT, Taub-Bolt, Eguchi-Hanson, Schwarzschild, Kerr and their AdS counterparts...) that were often discovered much later with  complicated ansatz, see \cite{1981PhRvD..24.1478L} where smooth Ricci-flat and compact Einstein PD metrics are classified. Extensions of these families more generally solve the Einstein-Maxwell equations and include known metrics such as Lebrun's scalar-flat metrics \cite{MR962489}. 

This family also contains families developing \textit{orbifold} singularities in the so-called AdS-Taub-Bolt family. It moreover contains continuous families of metrics exhibiting global \textit{collapsing} bubbling out (Ricci-flat) Taub-NUT or Schwarzschild metrics in the so-called AdS-Taub-NUT (or Pedersen's) metrics or AdS-Schwarzschild families. We will focus on cusp formation here.

\subsection*{Families of Poincaré-Einstein metrics forming cusps}

\subsubsection*{Degeneration in the family of AdS $C$-metrics}

It is now classical in the physics literature that a limit ``without rotation or twisting'' of the PD metrics leads to the well-known AdS $C$-metrics whose Ricci-flat versions were found by Levi-Civita \cite{levicivita} and Weyl \cite{Weyl:1917gp} in the 1910's(!). In this family, we first find a $2$-dimensional moduli space of smooth Poincaré-Einstein metrics on $\mathbb{R}^4$ containing the hyperbolic $4$-metric and whose limiting behaviors include metrics forming one or two cusps. A significant asymptotic quantity of Poincaré-Einstein metrics is the \textit{renormalized volume} defined in \cite{gra}. Despite the drastic degenerations presented in this article, the renormalized volume stays bounded.

\begin{thm}[Section \ref{section degeneration ads c metric}]\label{thm ads C metrics intro}
    There exists a smooth family of smooth Poincaré-Einstein metrics on $\mathbb{R}^4$ parametrized by an open region $\Omega$ in $\mathbb{R}^2$. Approaching some points at the boundary $\partial\Omega$, the metrics converge smoothly to the hyperbolic space or degenerate forming one or two codimension $2$ cusps. These cusps have asymptotic behaviors : $$dr_1^2+ae^{-r_1}d\theta_1^2+dr_2^2+ bd\theta_2^2 \;\;\text{ for }\;\;  r_1\in [0,+\infty),\; r_2\in[0,1], \; \theta_1,\theta_2\in[0,2\pi],$$ for $a,b>0$ in the bulk of the manifold, and $dr^2+ae^{-r}d\theta_1^2+bd\theta_2^2$ at conformal infinity with $r\in [0,+\infty)$. These examples have uniformly bounded renormalized volume.
    \end{thm}

An important question left open is the following one.
\begin{quest}
    Does there exist a continuous family of Poincaré-Einstein metrics forming cusps separating the manifold into a complete \emph{finite volume} piece  and another complete Poincaré-Einstein metric?
\end{quest}

\begin{rem}\label{remark all cusp rech infinity}
    Unfortunately, this is impossible in our family of metrics and there is little hope to find such a family of metrics explicitly given in coordinates. Indeed, in our case, one limit of such a degeneration has to be an Einstein metric with negative Ricci curvature and with at least one Killing vector field with finite length, which is impossible by Bochner's formula; see \cite{MR758664} 
    for instance.
\end{rem}

\subsubsection*{Degeneration in the Carter-Pleba\'nski family of metrics}

The limits ``without acceleration'' of the PD metrics consitute the \textit{Carter-Pleba\'nski} family of metrics. In Section \ref{section degeneration CP}, we exhibit a subfamily of smooth Poincaré-Einstein metrics with topology $\mathbb{CP}^2\backslash D^4$ forming cusp in some limits, and discuss how other topologies may be reached.

\subsubsection*{Degeneration in the full Pleba\'nski-Demia\'nski family of metrics}

In the full family of PD metrics, we also obtain cusps as in Theorem \ref{thm ads C metrics intro} which are this time ``twisted'' as in \eqref{separating cusp}. We additionally find families of smooth Poincaré-Einstein metrics on $\mathbb{R}^4$ where \textit{only} the conformal infinity degenerates in some limit.

\begin{thm}[Section \ref{section degeneration conformal infinity}]\label{degeneration infinity}
    There exists a smooth family of Poincaré-Einstein metrics on $\mathbb{R}^4$ whose conformal infinity approaches one of the following behaviors in some limit: for $a,b>0$
    \begin{itemize}
        \item A conical (edge) singularity: $ dr^2+ a^2r^2d\theta_1^2+b^2d\theta_2^2 $ on $(r,\theta_1,\theta_2)\in [0,1]\times [0,2\pi]\times [0,2\pi]$,
        \item A naked singularity: $ dr^2+ a^2r^6d\theta_1^2+b^2d\theta_2^2 $ on $(r,\theta_1,\theta_2)\in [0,1]\times [0,2\pi]\times [0,2\pi]$, or
        \item A cusp end: $ dr^2+ a^2e^{-4r}d\theta_1^2+b^2d\theta_2^2 $ on $(r,\theta_1,\theta_2)\in [0,+\infty]\times [0,2\pi]\times [0,2\pi]$.
    \end{itemize}
    While approaching these behaviors at conformal infinity, the metrics converge smoothly in the bulk metric in the pointed Cheeger-Gromov sense. These examples have uniformly bounded renormalized volume.
\end{thm}
\noindent These degenerations can occur in various limits that we describe in Section \ref{section degeneration conformal infinity} and Section \ref{section appendix possible asymptotics boundary}.

\vspace{-5pt}

\subsection*{Acknowledgments}
Carlos A. Alvarado and Daniel A. Santiago are grateful to the MIT Mathematics Department and Undergraduate Research Opportunities Program for their funding and support. They thank Tomasz Mrowka for his supervision and support. Tristan Ozuch would like to thank University of Münster for the amazing visiting conditions as a Young Research fellow at the Cluster of excellence, and Hans-Joachim Hein for stimulating discussions on the topic of this article during this visit.

\vspace{-5pt}

\section{The families of metrics considered}

\subsection{Pleba\'nski-Demia\'nski family of metrics}\label{section PD}

A ``Euclideanized'' Pleba\'nski-Demia\'nski (PD) metric has the following form 

\vspace{-5pt}

\small
\begin{equation}
\begin{aligned}
    g_{PD} = \frac{1}{(x-y)^2}&\left[-\frac{Q(y)}{1-a^2x^2y^2}(d\psi-ax^2d\varphi)^2 - \frac{1-a^2x^2y^2}{Q(y)}dy^2+ \frac{P(x)}{1-a^2x^2y^2}(d\varphi-ay^2d\psi)^2+\frac{1-a^2x^2y^2}{P(x)}dx^2\right] 
\end{aligned}\label{family PD with a}
\end{equation} 
\normalsize where $Q(y)$ and $P(x)$ are polynomials of degree $4$ which can be chosen depending on the value of $a\in\mathbb{R}$, physically understood as a \textit{rotation parameter}, so that $g_{PD}$ is an Einstein metric with $\Ric_{g_{PD}} = -3g_{PD}$ following the (Riemannian version of the) computations in \cite{PLEBANSKI197698}. Up to rescaling, we can assume $a \in \{0,1\}$. 
\\

Let us first consider the larger family with $a =1$ from which the other ones can be obtained from various limiting procedures. The Einstein condition \eqref{Einstein condition} with $\Lambda=-3$ is equivalent to $P$ and $Q$ having the form
\begin{equation}
    \begin{aligned}
    P(x) &= bx^4 + cx^3 + dx^2 + ex + b + 1 \text{ and }\\
    Q(y) &= (b+1)y^4+cy^3+dy^2+ey+b,
    \end{aligned}\label{polynoms general PD}
\end{equation}
for $b,c,d,e\in\mathbb{R}$ where we note the identity $Q(y) = P(y)+y^4-1$. The \textit{local} metric \eqref{family PD with a} is then Einstein \textit{and Riemannian} on ranges depending on roots of $P$ and $Q$. When ``closing-up'' at roots of $P$ and $Q$ it may have codimension $2$ cone-edge singularities (which we will avoid) or cusp ends which are discussed in Appendix \ref{section appendix possible asymptotics}. These metrics are moreover Poincaré-Einstein since they are conformal to a metric with boundary: the boundary is given by $\{x=y\}$ and the conformal factor is $\frac{1}{(x-y)^2}$. The \textit{conformal infinity} of these metrics is the conformal class of the metric induced on $\{x=y\}$ by $(x-y)^2g_{PD}$. We will see in different instances, especially in Section \ref{section degeneration conformal infinity} that these conformal infinities may degenerate. The possible degenerations of the conformal infinity are collected in the Appendix \ref{section appendix possible asymptotics boundary}. 

Without loss of generality, we can write $c = k_+ + k_-, e = k_+-k_-$ in \eqref{polynoms general PD}, in which the set of eigenvalues of the $\pm$-selfdual part of Weyl curvature $W_{g_{PD}}$ is proportional to $ \frac{k_\pm}{(1 \pm x y)^3} (2,-1,-1) $. The pointwise norm of the Riemannian tensor of $g_{PD}$ is given by

\vspace{-8pt}

$$||\Rm_{g_{PD}}||^2 = 24 + 24(x - y)^6\left( \frac{k_+^2}{(1 + x y)^6} + \frac{k_-^2}{(1-xy)^6}\right).$$

\vspace{-5pt}

The volume element in these coordinates is $ \frac{-1 + x^2 y^2}{(x - y)^4}dxdyd\varphi d\psi $ and one checks that $ \|W_{g_{PD}}\|_{L^2(g_{PD})} $ is finite for the domains we consider, hence, by \cite{MR1825268}, the renormalized volume is controlled for our examples.

\subsection{Non rotating limit: AdS C-metrics}\label{section AdS C metric}
For this section, we 
%use the notation and the
%study the Riemannian %version of the metrics in 
will follow
\cite{PhysRevD.92.044058}  and will adopt their notation. Our study and goals are purely geometric and differ from theirs.
The AdS C-metrics are obtained from the general Pleba\'nski-Demia\'nski family \eqref{family PD with a} by taking the \textit{non-rotating limit} $a \to 0$. These metrics are the Riemannian analogues of the metrics considered in \cite{PhysRevD.91.064014}, and have the form

\vspace{-5pt}

\begin{equation}
    g_{C} = \frac{1}{(x-y)^2} \left[ -Q(y)d\psi^2-\frac{dy^2}{Q(y)}+\frac{dx^2}{P(x)}+P(x)d\varphi^2 \right]\label{C-metric}
\end{equation}
where we will assume that $Q$ and $P$ are parametrized by two variables $\mu,\nu$ as 
\begin{equation}
    \begin{aligned}
    P(x) &= (1+x)\left(1+\nu x+\mu x^2\right),\text{ and }\\
    Q(y) &= y\left[1+\nu+(\mu+\nu)y+\mu y^2\right],
    \end{aligned}\label{polynoms general C-metric}
\end{equation}
This ensures Einstein condition \eqref{Einstein condition} is satisfied. The pointwise norm of the Riemannian tensor of $g_C$ is given by
$$\|\Rm_{g_C}\|_{g_C}^2 = 24+12(x-y)^6\mu^2,$$
and more precisely, one has $ \Ric(g_C)=-3g_C $ and the eigenvalues of both the selfdual and anti-selfdual parts of $W_{g_C}$ are equal to $\frac{\mu}{4} (y-x)^3(2,-1,-1)$ which as expected go to zero as $x\to y$. Moreover, when $\mu=0$, the metric is \textit{locally hyperbolic}. A direct computation ensures again that $\|W_{g_C}\|_{L^2(g_C)}^2$ is bounded. In particular, from \cite{MR1825268}, these examples have bounded \textit{renormalized volume}.

%$$Kr_C = 4(-2P(x)+2Q(y)+(x-y)(P'(x)+Q'(y)))^2+(2P(x)-2Q(y)+2(y-x)G'(x)+(x-y)^2G''(x))^2$$
%$$+(-2 P(x)+2Q(y)+(x-y)(2Q'(y)+(x-y)Q''(y)))^2$$

\subsection{Non accelerating limit: Carter-Pleba\'nski metrics}

The Carter-Pleba\'nski family of metrics is a special limit of the Pleba\'nski-Demia\'nski family of metrics \eqref{family PD with a} after a change of coordinates. To do this, start from \eqref{family PD with a} in the coordinates of \cite{1981PhRvD..24.1478L} and perform a rescaling by $b>0$ (acceleration parameter) of coordinates as in \cite[Section 2.2]{griffiths2006new}, which yields the following metric:
\begin{equation}
    g_{PD} = \frac{1}{(1-bpq)^2} \left[\frac{p^2-q^2}{\mathcal{P}_b(p)}dp^2 + \frac{q^2-p^2}{\mathcal{Q}_b(q)}dq^2 + \frac{\mathcal{P}_b(p)}{p^2-q^2}\left(d\tau+q^2d\sigma\right)^2 + \frac{\mathcal{Q}_b(q)}{q^2-p^2}\left(d\tau+p^2d\sigma\right)^2\right]\label{Family PD other coord}
\end{equation}
 for polynomials $\mathcal{P}_b$ and $\mathcal{Q}_b$ depending on $b>0$ chosen to satisfy \eqref{Einstein condition} with $\Lambda = -3$.  Taking the ``no acceleration limit'' $b \to 0$ as in \cite[Section 5]{griffiths2006new}, we obtain from \eqref{Family PD other coord} the metric
\begin{equation}
    g_{CP}:= \frac{p^2-q^2}{\mathcal{P}(p)}dp^2 + \frac{q^2-p^2}{\mathcal{Q}(q)}dq^2 + \frac{\mathcal{P}(p)}{p^2-q^2}\left(d\tau+q^2d\sigma\right)^2 + \frac{\mathcal{Q}(q)}{q^2-p^2}\left(d\tau+p^2d\sigma\right)^2,\label{Family CP}
\end{equation}
where the limiting polynomials $\mathcal{P}$ and $\mathcal{Q}$ are of the form:
\begin{equation} 
\begin{aligned}
    \mathcal{P}(p) &= p^4 + E^2p^2 -2Np + \alpha  \text{ and }\\
    \mathcal{Q}(q) &= q^4 + E^2q^2 -2Mq + \alpha
\end{aligned}
\end{equation}
following the notations of \cite{Martelli_2013} for some real numbers $E$, $M$, $N$ and $\alpha$.

We will consider intervals where $\mathcal{P}(p) \leqslant 0$ and $\mathcal{Q}(q)\geqslant 0$. This time, the range in $p$ will be compact of the form $[p_-,p_+]$ for $p_\pm$ roots of $\mathcal{P}$ and the range in $q$ will be of the form $[q_+,+\infty)$ for $q_+$ root of $\mathcal{Q}$.

This metric is Poincaré-Einstein and as $q \to + \infty$ (the infinity in these coordinates), the metric looks like
\begin{equation}
    g_{CP} \approx \frac{dq^2}{q^2} + q^2\left( -\frac{dp^2}{\mathcal{P}(p)}-\mathcal{P}(p)d\sigma^2+(d\tau + p^2d\sigma)^2\right)
\end{equation}
so the metric at conformal infinity is $-\frac{dp^2}{\mathcal{P}(p)}-\mathcal{P}(p)d\sigma^2+(d\tau + p^2d\sigma)^2$.

%\subsection{Nonexpanding limit: AdS B-metrics with rotation}

%Tentative metric:
%\begin{equation}
    %g=(\gamma^2-p^2)\left(\frac{dp^2}{\mathcal{P}(p)}+\mathcal{Q}(q) d\varphi^2 + \frac{dq^2}{\mathcal{Q}(q)}\right) + \frac{\mathcal{P}(p)}{\gamma^2-p^2}(d\psi - 2 \gamma q d\varphi)^2
%\end{equation}
%with $ \mathcal{Q}(q) = \epsilon_0 -\epsilon_2 q^2 $ and $ \mathcal{P}(p) = (-\epsilon_2\gamma^2 - 3 \gamma^4)p^2-2np - (\epsilon_2 +6\gamma^2) -p^4 $ with $\epsilon_i\in\{\pm 1\}$

%\todo[inline]{To describe}

%\begin{rem}
   % Another limit exists when the ``acceleration'' is set to zero in a suitable system of coordinates. This limit is known as \textit{Carter-Pleba\'nski's family}, but will not be discussed further in the present paper.
%\end{rem}

\section{Degenerations of AdS C-metrics}\label{section degeneration ads c metric}

In this section, we study a specific $2$-dimensional family of AdS C-metrics on $\mathbb{R}^4$ forming one or two cusps in different limits. The cusps forming here effectively separate the manifold into two or three Poincaré-Einstein metrics with cusps ends in their bulk and their conformal infinities. We prove Theorem \ref{thm ads C metrics intro}.

As in Section \ref{section AdS C metric} we consider the metric \eqref{C-metric}
where $Q(y) = y\left[1+\nu+(\mu+\nu)y+\mu y^2\right]$ and $P(x) = (1+x)\left(1+\nu x+\mu x^2\right)$. The roots of $P$ and $Q$ respectively are as follows  \begin{equation} \label{ADSRoots}
\begin{aligned}
x_0&= -1,  &x_{\pm}= \frac{-\nu \pm \sqrt{\nu^2-4\mu}}{2\mu}, \text{ and }\\
y_0&= 0,  &y_{\pm}=  \frac{-(\mu+\nu) \pm \sqrt{(\mu+\nu)^2-4\mu(1+\nu)}}{2\mu}.
\end{aligned}
\end{equation}

%\todo[inline]{We need to give the range of coordinates $-1\leqslant x<y\leqslant0$ and justify briefly that we have the right signs for the values of $\mu$ and $\nu$ we choose.}

%\todo[inline]{We need to give the periods of $\varphi$, $\psi$ ensuring the regularity of the metric. Use Proposition \ref{regularity criterion} in the appendix.}

\noindent  In order to approach metrics with cusp ends in this family by smooth metrics, we consider the case when $x_{\pm}, y_{\pm}$ are complex conjugate roots which we will let approach a real double root -- leading to a cusp degeneration by Section \ref{section approaching cusp}. In the $(\mu,\nu)$ plane, this condition means that $(\mu,\nu)$ lies in the region bounded by the curves $\nu = 2 \sqrt{\mu}$ and $\nu = \mu - 2 \sqrt{\mu}$.

We then consider $-1<x<y<0$ where the conformal infinity is at $\{x=y\}$, see Figure \ref{range x y c metric}. For the metric to be smooth, we require that $\frac{1-\nu + \mu}{2}\varphi$ and $ \frac{1+\nu}{2}\psi$ be $2\pi$-periodic, see Proposition \ref{regularity criterion}. We further impose that $\mu>\max(\nu/2,-\nu)$. This corresponds to forcing the real part of $x_\pm$ and $y_\pm$ to be in $(-1,0)$, this way the double root degeneration (when the imaginary part of the roots tends to zero) happens where the metric is defined and is geometrically meaningful. 
We end up with the region $D4$ in \cite{PhysRevD.92.044058} shaded in Figure \ref{figure admissible mu nu parameters} in the $\mu,\nu$ plane bounded by the curves $\nu = 2 \sqrt{\mu}$, $\nu = \mu - 2 \sqrt{\mu}$, $\nu = 2\mu$ and $\nu = -\mu$.

%\todo{recall the range of possible double roots in $P$ and $Q$} Will add below 

%$x_{\pm} = \frac{-1}{4}, y_{\pm} = \frac{-3}{4}$. 
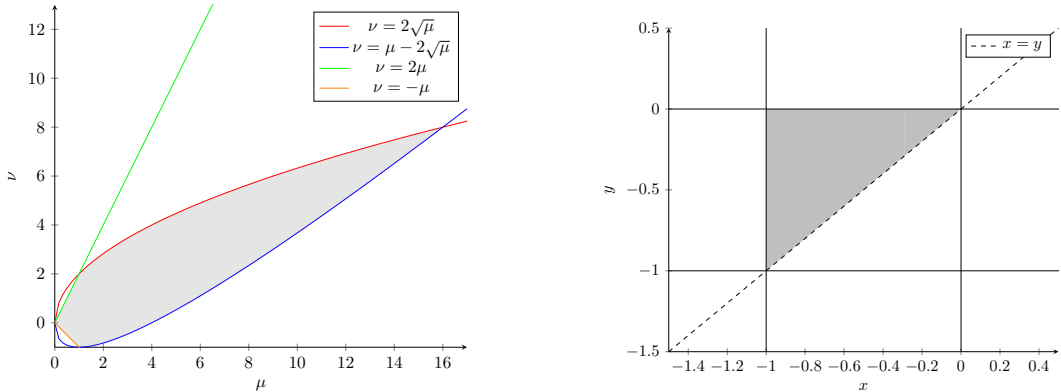
\begin{figure}[H]
\centering
\begin{subfigure}[t]{0.45\textwidth}
\begin{center}
\resizebox{0.8\linewidth}{!}{
\begin{tikzpicture}

\begin{axis}[
    axis lines = left,
    xlabel = \(\mu\),
    ylabel = {\(\nu\)},
    ymin = -1,
    ymax = 13,
    xmin = 0,
    xmax = 17,
]
%Below the red parabola is defined
\addplot [
    name path = g,
    domain=0:17, 
    samples=100, 
    color=red,
]
{2*(x^0.5)};
\addlegendentry{\(\nu=2\sqrt{\mu}\)}
%Here the blue parabola is defined
\addplot [
    name path=f, 
    domain=0:17, 
    samples=100, 
    color=blue,
    ]
    {x-2*(x^0.5)}; % x-(2*(x**.5)
\addlegendentry{\(\nu=\mu-2\sqrt{\mu}\)}

\addplot [
    name path=h,
    domain=0:17, 
    samples=100, 
    color=green
    ]
    {2*x};
\addlegendentry{\(\nu=2\mu\)}
\addplot [
    name path = i,
    domain=0:17, 
    samples=100, 
    color=orange
    ]
    {-x};
\addlegendentry{\(\nu=-\mu\)}
\addplot [
        thick,
        color=gray,
        fill=gray, 
        fill opacity=0.2
    ]
    fill between[
        of=h and i,
        soft clip ={domain=0:1},
    ];

\addplot [
        thick,
        color=gray,
        fill=gray, 
        fill opacity=0.2
    ]
    fill between[
        of=f and g,
        soft clip ={domain=1:16},
    ];

\end{axis}

\end{tikzpicture}}
\end{center}

\vspace{-5pt}

\caption{Admissible $(\mu,\nu)$ parameters (shaded). }\label{figure admissible mu nu parameters}
\end{subfigure}
\begin{subfigure}[t]{0.45\textwidth}
\begin{center}
\resizebox{0.8\linewidth}{!}{
\begin{tikzpicture}

\begin{axis}[
    axis lines = left,
    xlabel = \(x\),
    ylabel = {\(y\)},
    ymin = -1.5,
    ymax = 0.5,
    xmin = -1.5,
    xmax = 0.5,
]
\addplot [
    name path=g,
    domain=-2:1, 
    samples=100, 
    color=black,
    dashed
]
{x};
\addlegendentry{\(x=y\)}
\addplot [
    name path =f,
    domain=-2:1, 
    samples=100, 
    color=black,
]
{0};
\addplot [
    domain=-2:1, 
    samples=100, 
    color=black,
]
{-1};

\draw (0,-2 |- current axis.south) -- (0,1 |- current axis.north);
\draw (-1, -2  |- current axis.south) -- (-1,1 |- current axis.north);
\addplot [
        thick,
        color=gray,
        fill=gray, 
        fill opacity=0.5
    ]
    fill between[
        of=f and g,
        soft clip ={domain=-1:-0.289},
    ];
\addplot [
        thick,
        color=gray,
        fill=gray, 
        fill opacity=0.5
    ]
    fill between[
        of=f and g,
        soft clip ={domain=-0.289:0},
    ];
\end{axis}

\end{tikzpicture}
}
\end{center}

\vspace{-5pt}

\caption{Range of coordinates $(x,y)$ (shaded), the range for $\varphi$, $\psi$ is a torus determined by Proposition \ref{regularity criterion}}\label{range x y c metric}
\end{subfigure}

\vspace{-5pt}

\caption{Parameter ranges considered. The dashed $\{x=y\}$ is the conformal infinity.}\label{figure C-metric}

\end{figure}

\vspace{-5pt}

\begin{rem}\label{remark C-metric mu nu=0}
    In the limit $(\mu, \nu) \to 0$ from our region shaded in Figure \ref{figure admissible mu nu parameters}, our metrics converge smoothly to the hyperbolic $4$-space. Indeed, the metric is already locally hyperbolic by our curvature computations and the change of variables 
$x = -\sin^2\left((u-\pi)/2\right), $
at the conformal infinity $\{x=y\}$, the restriction of the metric $(x-y)^2g_{CLT}$ with $\mu=\nu=0$ takes the form

\vspace{-10pt}

$$du^2 + \cos \left(\frac{u-\pi}{2} \right)^2d\varphi^2+ \sin \left(\frac{u-\pi}{2} \right)^2d\psi^2. $$

\vspace{-5pt}

Thus we recover the metric of the round $3$-sphere in Hopf's coordinates since $\varphi$ and $\psi$ are $4\pi$-periodic. This in particular ensures that the topology we consider is $\mathbb{R}^4$.
\end{rem}

From (\ref{ADSRoots}) we see that, for $(\mu,\nu)$ in the shaded region in Figure \ref{figure admissible mu nu parameters}, if one of $P,Q$ has a double root, then $(\mu,\nu)$ lies on at least one of the boundary curves $\nu = 2 \sqrt{\mu}$ or $\nu = \mu -2\sqrt{\mu}$ respectively in blue and red in Figure \ref{figure C-metric}, see the first two columns of Figure \ref{configurations of double roots} for the associated polynomials and geometric representation. The intersection of these curves, $(\mu,\nu) = (16,8)$, is the unique case when $P$ and $Q$ have double roots at $x = \frac{-1}{4}$ and $y=\frac{-3}{4})$ as described in Figures \ref{two cusps} and \ref{two double roots Cmetric}  respectively leading to two cusps dividing the manifold in three regions, while the point $(\mu,\nu)=(0,0)$ corresponds to hyperbolic $4$-space from Remark \ref{remark C-metric mu nu=0}. The possible double roots of $P$ and $Q$ respectively lie in the intervals $(-1,\frac{-1}{4}]$ and $ [\frac{-3}{4},0)$.
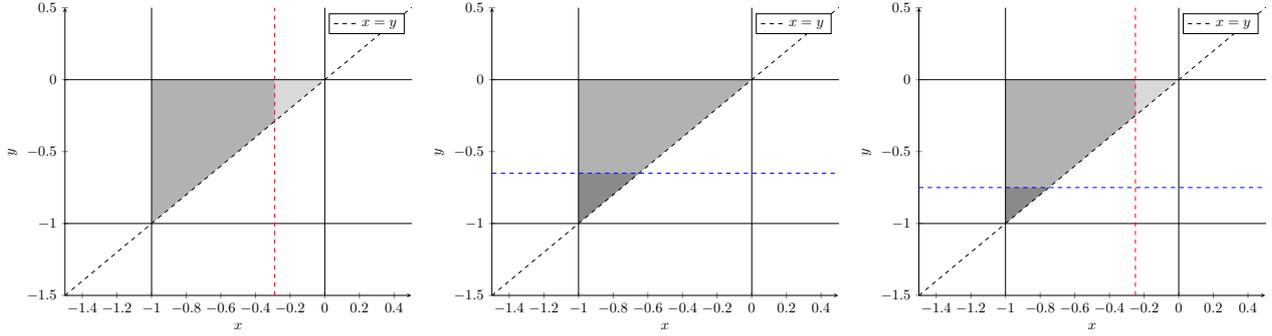
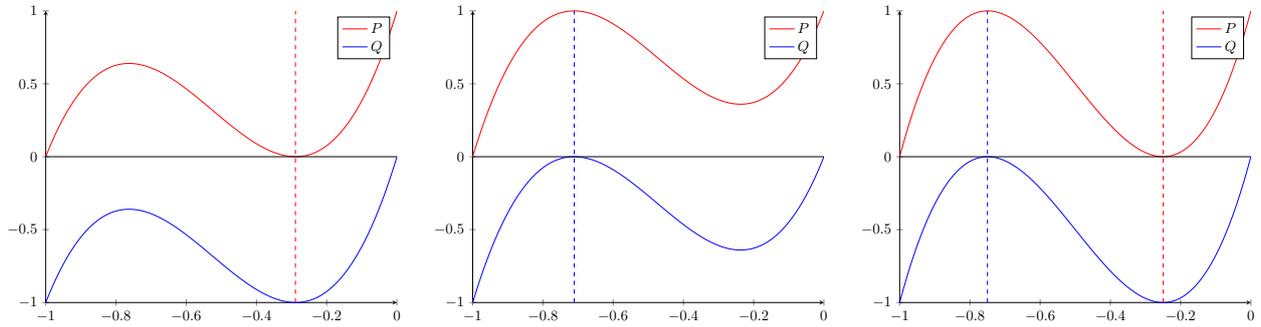
\begin{figure}[H]
\begin{subfigure}[t]{0.32\textwidth}
\begin{center}
\resizebox{\linewidth}{!}{
\begin{tikzpicture}

\begin{axis}[
    axis lines = left,
    xlabel = \(x\),
    ylabel = {\(y\)},
    ymin = -1.5,
    ymax = 0.5,
    xmin = -1.5,
    xmax = 0.5,
]
\addplot [
    name path=g,
    domain=-2:1, 
    samples=100, 
    color=black,
    dashed
]
{x};
\addlegendentry{\(x=y\)}
\addplot [
    name path =f,
    domain=-2:1, 
    samples=100, 
    color=black,
]
{0};
\addplot [
    domain=-2:1, 
    samples=100, 
    color=black,
]
{-1};

\draw (0,-2 |- current axis.south) -- (0,1 |- current axis.north);
\draw (-1, -2  |- current axis.south) -- (-1,1 |- current axis.north);
\draw[red, dashed] (-0.289, -2  |- current axis.south) -- (-0.289,1 |- current axis.north);

\addplot [
        thick,
        color=gray,
        fill=gray, 
        fill opacity=0.6
    ]
    fill between[
        of=f and g,
        soft clip ={domain=-1:-0.289},
    ];
\addplot [
        thick,
        color=gray,
        fill=gray, 
        fill opacity=0.3
    ]
    fill between[
        of=f and g,
        soft clip ={domain=-0.289:0},
    ];
\end{axis}

\end{tikzpicture}
}
\end{center}

\vspace{-10pt}

\caption{Double root (red) in P separating the manifold.}
\end{subfigure}
\begin{subfigure}[t]{0.32\textwidth}
\begin{center}
\resizebox{\linewidth}{!}{
\begin{tikzpicture}

\begin{axis}[
    axis lines = left,
    xlabel = \(x\),
    ylabel = {\(y\)},
    ymin = -1.5,
    ymax = 0.5,
    xmin = -1.5,
    xmax = 0.5,
]
\addplot [
    name path=g,
    domain=-2:1, 
    samples=100, 
    color=black,
    dashed
]
{x};
\addlegendentry{\(x=y\)}
\addplot [
    name path =f,
    domain=-2:1, 
    samples=100, 
    color=black,
]
{0};
\addplot [
    domain=-2:1, 
    samples=100, 
    color=black,
]
{-1};
\addplot [
    name path =t,
    domain=-2:1, 
    samples=100, 
    color=blue,
    dashed
]
{-0.65};

\draw (0,-2 |- current axis.south) -- (0,1 |- current axis.north);
\draw (-1, -2  |- current axis.south) -- (-1,1 |- current axis.north);

\addplot [
        thick,
        color=gray,
        fill=gray, 
        fill opacity=0.8
    ]
    fill between[
        of=g and t,
        soft clip ={domain=-1:-0.65},
    ];
\addplot [
        thick,
        color=gray,
        fill=gray, 
        fill opacity=0.6
    ]
    fill between[
        of=f and g,
        soft clip ={domain=-1:0},
    ];
\end{axis}

\end{tikzpicture}
}
\end{center}

\vspace{-10pt}

\caption{Double root (blue) in $Q$ separating the manifold.}
\end{subfigure}
\begin{subfigure}[t]{0.32\textwidth}
\begin{center}
\resizebox{\linewidth}{!}{
\begin{tikzpicture}

\begin{axis}[
    axis lines = left,
    xlabel = \(x\),
    ylabel = {\(y\)},
    ymin = -1.5,
    ymax = 0.5,
    xmin = -1.5,
    xmax = 0.5,
]
\addplot [
    name path=g,
    domain=-2:1, 
    samples=100, 
    color=black,
    dashed
]
{x};
\addlegendentry{\(x=y\)}
\addplot [
    name path =f,
    domain=-2:1, 
    samples=100, 
    color=black,
]
{0};
\addplot [
    domain=-2:1, 
    samples=100, 
    color=black,
]
{-1};

\addplot [
    name path =t,
    domain=-2:1, 
    samples=100, 
    color=blue,
    dashed
]
{-0.75};

\draw (0,-2 |- current axis.south) -- (0,1 |- current axis.north);
\draw (-1, -2  |- current axis.south) -- (-1,1 |- current axis.north);
\draw[red,dashed] (-0.25, -2  |- current axis.south) -- (-0.25,1 |- current axis.north);

\addplot [
        thick,
        color=gray,
        fill=gray, 
        fill opacity=0.6
    ]
    fill between[
        of=f and g,
        soft clip ={domain=-1:-0.25},
    ];
\addplot [
        thick,
        color=gray,
        fill=gray, 
        fill opacity=0.3
    ]
    fill between[
        of=f and g,
        soft clip ={domain=-0.25:0},
    ];
    
\addplot [
        thick,
        color=gray,
        fill=gray, 
        fill opacity=0.8
    ]
    fill between[
        of=g and t,
        soft clip ={domain=-1:-0.75},
    ];
\end{axis}

\end{tikzpicture}
}
\end{center}

\vspace{-10pt}

\caption{Double root in $P$ (red) and $Q$ (blue) separating in three.}\label{two cusps}
\end{subfigure}

\begin{subfigure}[t]{0.32\textwidth}
\centering
\resizebox{\linewidth}{!}{
\begin{tikzpicture}
\begin{axis}[
    axis lines = left,
    xlabel = \(\),
    ylabel = {\(\)},
    ymin = -1,
    ymax = 1,
    xmin = -1,
    xmax = 0,
]
%Below the red parabola is defined
\addplot [
    domain=-1:0, 
    samples=100, 
    color=red,
]
{(1+x)*(1+6.9282032*x+12*x^2)};
\addlegendentry{\(P\)}
%Here the blue parabola is defined
\addplot [
    domain=-1:0, 
    samples=100, 
    color=blue,
    ]
    {(1+x)*(1+6.9282032*x+12*x^2)-1};
\addlegendentry{\(Q\)}
\addplot [
    domain=-1:0, 
    samples=100, 
    color=black,
    ]
    {0};
\draw[red, dashed] (-0.289, -1) -- (-0.289,1);

\end{axis}
\end{tikzpicture}}
\caption{Double root only in P, \\
$\nu = 2\sqrt{\mu}$.}
\end{subfigure}
\begin{subfigure}[t]{0.32\textwidth}
\centering
\resizebox{\linewidth}{!}{
\begin{tikzpicture}
\begin{axis}[
    axis lines = left,
    xlabel = \(\),
    ylabel = {\(\)},
    ymin = -1,
    ymax = 1,
    xmin = -1,
    xmax = 0,
]
%Below the red parabola is defined
\addplot [
    domain=-1:0, 
    samples=100, 
    color=red,
]
{(1+x)*(1+5.07179676972*x+12*x^2)};
\addlegendentry{\(P\)}
%Here the blue parabola is defined
\addplot [
    domain=-1:0, 
    samples=100, 
    color=blue,
    ]
    {(1+x)*(1+5.07179676972*x+12*x^2)-1};
\addlegendentry{\(Q\)}
\addplot [
    domain=-1:0, 
    samples=100, 
    color=black,
    ]
    {0};

\draw[blue, dashed] (-0.711, -1) -- (-0.711,1);
\end{axis}
\end{tikzpicture}}\caption{Double root only in Q, \\
$\nu = \mu-2\sqrt{\mu}$.}\label{two double roots}
\end{subfigure}
\begin{subfigure}[t]{0.32\textwidth}
\centering
\resizebox{\linewidth}{!}{
\begin{tikzpicture}
\begin{axis}[
    axis lines = left,
    xlabel = \(\),
    ylabel = {\(\)},
    ymin = -1,
    ymax = 1,
    xmin = -1,
    xmax = 0,
]
%Below the red parabola is defined
\addplot [
    domain=-1:0, 
    samples=100, 
    color=red,
]
{(1+x)*(1+8*x+16*x^2)};
\addlegendentry{\(P\)}
%Here the blue parabola is defined
\addplot [
    domain=-1:0, 
    samples=100, 
    color=blue,
    ]
    {(1+x)*(1+8*x+16*x^2)-1};
\addlegendentry{\(Q\)}
\addplot [
    domain=-1:0, 
    samples=100, 
    color=black,
    ]
    {0};

\draw[red, dashed] (-0.25, -1) -- (-0.25,1);
\draw[blue, dashed] (-0.75, -1) -- (-0.75,1);
\end{axis}
\end{tikzpicture}}
\caption{Double root in both P and Q, \\$(\mu,\nu) = (16,8)$.}\label{two double roots Cmetric}
\end{subfigure}

\vspace{-5pt}

\caption{Different configurations of double roots.}\label{configurations of double roots}
\end{figure}

\section{Degenerations in Carter-Pleba\'nski family of metrics}\label{section degeneration CP}

In this section, we indicate how to find families of metrics forming cusps with different topologies. We take the simplest example here on $\mathbb{CP}^2\backslash D^4$ with conformal infinity $\mathbb{S}^3$. We follow \cite[Sections 2.1, 2.2 and 2.3]{Martelli_2013} for our regularity conditions: we impose $\tau$ and $\sigma$ to be as \cite[Sections 2.13 and  2.17]{Martelli_2013}. This also requires $N=M$, which is equivalent to the metric being \textit{self-dual}, and forcing $\mathcal{P}=\mathcal{Q}$. 

We will moreover parametrize our polynomial by the roots and looking for a metric with a cusp, we will consider a polynomial with a double root: for $p_3,p_4,p_0\in \mathbb{R}$ (following notations of \cite{Martelli_2013}):
\begin{equation}
    \mathcal{P}(p) = (p-p_3)(p-p_4)(p-p_0)^2.\label{polynomial CP}
\end{equation}
We will then consider the range $(p,q)\in [p_3,p_4]\times [p_4,+\infty]$, where the associated metric is indeed Riemannian.

\begin{figure}[h]
\begin{subfigure}[b]{0.48\textwidth}
\centering
\resizebox{0.8\linewidth}{!}{
\begin{tikzpicture}
\begin{axis}[
    axis lines = left,
    xlabel = \(p  \;\;\text{ or }\;\; q\),
    ylabel = {\(\)},
    ymin = -0.3,
    ymax = 0.5,
    xmin = -1.5,
    xmax = 2.5,
]
%Below the red parabola is defined
\addplot [
    domain=-2:2, 
    samples=100, 
    color=red,
]
{(x^2)*(x-1)*(x+1)};
\addlegendentry{\(\mathcal{P} = \mathcal{Q}\)}
%Here the blue parabola is defined
\addplot [
    domain=-1:1, 
    samples=100, 
    color=orange,
    line width=0.7mm
    ]
    {0};
\addlegendentry{\((p_3,p_4)\)}
\addplot [
    domain=1:3, 
    samples=100, 
    color=green,
    line width=0.7mm
    ]
    {0};
\addlegendentry{\((p_4,\infty)\)}
\addplot [
    domain=-2:-1, 
    samples=100, 
    color=black,
    ]
    {0};

\draw [color=red, dashed] (0,-2 |- current axis.south) -- (0,2 |- current axis.north);

\end{axis}
\end{tikzpicture}}
\caption{Example of double root in $\mathcal{P}$ with $N=M$. Intervals where $\mathcal{P}$ and $\mathcal{Q}$ are defined are highlighted.}\label{figure range polynomial carter}
\end{subfigure}
\begin{subfigure}[b]{0.48\textwidth}
\begin{center}
\resizebox{0.8\linewidth}{!}{
\begin{tikzpicture}

\begin{axis}[
    axis lines = left,
    xlabel = \(q\),
    ylabel = {\(p\)},
    ymin = -2,
    ymax = 2,
    xmin = 0.5,
    xmax = 3,
]
%Below the red parabola is defined
\addplot [
    name path=f,
    domain=-2:3, 
    samples=100, 
    color=black,
]
{-1};

\addplot [
    domain=-2:3, 
    samples=0, 
    color=black,
]
{-1};

\addplot [
    name path=g,
    domain=-2:3, 
    samples=100, 
    color=black,
    ]
    {1}; % x-(2*(x**.5)

\draw (1,-2 |- current axis.south) -- (1,2 |- current axis.north);

\draw (-1,-2 |- current axis.south) -- (-1,2 |- current axis.north);

\addplot [
    name path=dr,
    domain=-2:3, 
    samples=100, 
    color=red,
    dashed
    ]
    {0}; % x-(2*(x**.5)

\draw [color=orange,
    line width=2mm] (0.5,-1) -- (0.5,1);

\addplot[
    name path=fd,
    domain=1:3, 
    samples=100, 
    color=green,
    line width=2mm
    ]
    {-2};

\addplot [
        thick,
        color=gray,
        fill=gray, 
        fill opacity=0.3
    ]
    fill between[
        of=f and g,
        soft clip ={domain=1:3},
    ];

\end{axis}

\end{tikzpicture}}
\end{center}

\vspace{-7pt}

\caption{Double root in $\mathcal{P}$ (in blue) extending infinity $q\to +\infty$. Range in $(p,q)$ shaded.}
\end{subfigure}

\vspace{-5pt}

\caption{Polynomial and range of coordinates.}
\end{figure}
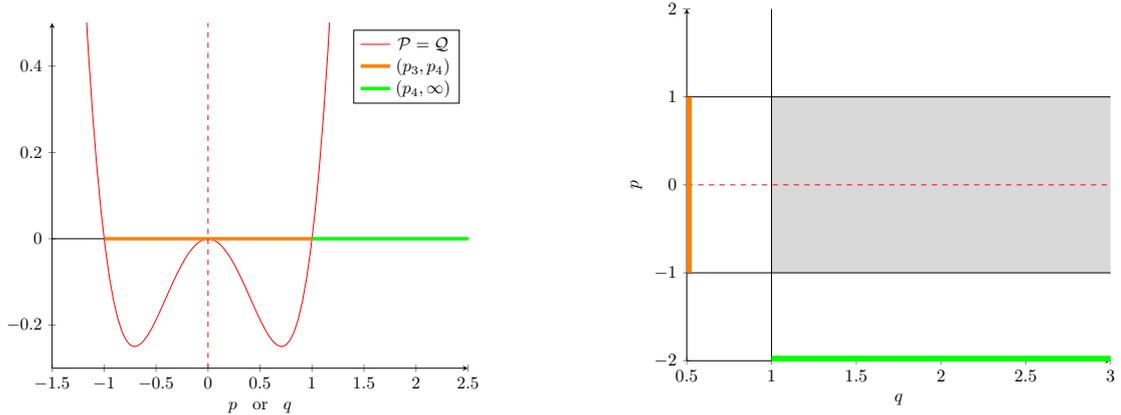

\begin{rem}
    Recall that from Remark \ref{remark all cusp rech infinity} we cannot have a double root in $\mathcal{Q}$ on $(p_4,+\infty)$. All we will find instead is a double root of $\mathcal{P}$ on $(p_3,p_4)$ corresponding to a cusp in the manifold extending to infinity. 
\end{rem}

We need our double root $p_0$, to lie in $(p_3,p_4)$ so that it is reflected in our metric. Since the sum of the roots is $0$ (the cubic coefficient of the polynomial is zero), $p_0 = -\frac{p_3+p_4}{2}$ and so $p_0\in(p_3,p_4)$ imposes 

\vspace{-5pt}

\begin{equation}
    p_3 < 0 < p_4,\; \text{ and  } \; \frac{1}{3}|p_3| < |p_4| < 3|p_3|.
\end{equation}
We can find this polynomial \eqref{polynomial CP} as a limit of polynomials with two complex conjugate roots: for $\epsilon\geqslant 0$
\begin{equation}
    \mathcal{P}_\epsilon(p) = (p-p_3)(p-p_4)\Big(\Big(p+\frac{1}{2}(p_3+p_4)\Big)^2+\epsilon^2\Big)
\end{equation}
where we get the double root mentioned above when $\epsilon \to 0$ and we also let $\mathcal{Q}_\epsilon = \mathcal{P}_\epsilon$ to satisfy the above regularity condition of \cite{Martelli_2013}. Since the roots of the polynomials are the same, the intervals in which these are defined stay the same. Geometrically, in the limit $\epsilon\to 0$, the metrics \eqref{Family CP} associated to $\mathcal{Q}_\epsilon = \mathcal{P}_\epsilon$ develop a cusp along $\{p=-(p_3+p_4)/2\}$ separating the manifold in two parts by an argument similar to Section \ref{section one double root}.
\\

The topology of the manifold is that of $\mathbb{CP}^2$ minus a ball and the conformal infinity is $\mathbb{S}^3$. The ``bolt'' of the metric is reached at $[p_3,p_4]\times\{q=1\}$ which is a codimension $2$ submanifold (a $2$-sphere) because of the degeneration of the metric there, see \cite{Martelli_2013} or the discussion in Section \ref{section appendix possible asymptotics}.

\begin{rem}
It is likely possible to obtain infinitely many different topologies from the Carter-Pleba\'nski family of metrics by having a larger and larger ``self-intersection'' for the $2$-sphere while obtaining a conformal infinity $\mathbb{S}^3/\mathbb{Z}_k$ for $\mathbb{Z}_k$ a cyclic subgroup of $SU(2)$ acting freely on $\mathbb{S}^3$. See \cite[Section 5.1]{CHEN2010207} for a discussion of the regularity conditions and possible topologies. In the larger Pleba\'nski-Demia\'nski family of metrics, we believe that there is also a large class of additional possible topologies, with two ``bolts'' (and a ``NUT''). The conformal infinity, could this time be an arbitrary lens space. See \cite[Section 5.2]{CHEN2010207} for a discussion of the regularity conditions and possible topologies.
\end{rem}

%To get this double root given by $\mathcal{P}(p)=(p-p_3)(p-p_4)(p+\frac{1}{2}(p_3+p_4))$\todo{$(p+\frac{1}{2}(p_3+p_4))^2$ ? So that there is a double root at $p_0=\frac{1}{2}(p_3+p_4)$}, we need $p_3 < 0 <  p_4$ (equality strict?), that $|p_3|<3|p_4|$ and $|p_4|<3|p_3|$ to ensure $p_0 \in (p_3,p_4)$ and any $k < 0$ such that $Q(x) = P(x) + kx$. 

%To see our conditions on the root, since the sum of the roots is 0, we have that $p_0 = -\frac{1}{2}(p_3+p_4)$. Since we need $p_0 \in (p_3,p_4)$ so that the root is reflected in our metric. If both of our terms were positive or negative, $p_0$ would have the opposite sign so it can't be in $(p_3,p_4)$. Similarly, if either of the conditions about the magnitude of the roots doesn't hold, $p_0$ won't be in $(p_3,p_4)$. 

%The reason why we have that $M - N > 0$ is because it ensures that $p_4 < q_+$. To see this, we know that $p_4$ and $q_+$ are the greatest roots of each $P$ and $Q$ respectively and are both greater than 0. We also have that after each of these roots, the polynomials are positive. Then we have that if $p_4 < q_+$, then $Q(q_+) - P(q_+) = kq_+$ and since the left hand side is less than 0, it implies $k < 0$. Similarly, if $q_+ < p_4$, then we have $k > 0$. 

\section{Degenerations in the Pleba\'nski-Demia\'nski family of metrics}\label{section degeneration conformal infinity}

%\todo[inline]{I think that generally, the pictures are too large, it would be good to have them side by side when it makes sense and maybe have other ones smaller or with text continuing on their side. (the shorter the article, the easier it is to publish!)}

We will now turn to the general PD family of metrics. The above degenerations of Sections \ref{section degeneration ads c metric} and \ref{section degeneration CP} can be found in the full family of Pleba\'nski-Demia\'nski, but we focus on exhibiting new behaviors of complete metrics whose conformal infinities develop unexpected types of singularities. We prove Theorem \ref{degeneration infinity}. 

In this section we consider a subfamily of metrics in \eqref{family PD with a} with $a=1$, parametrizing our polynomials as

\vspace{-5pt}

\begin{equation}
    P_{\infty}(x) = 
%-\frac{12}{11}(x-\frac{1}{2})(x-1)^2(x-\frac{1}{6})$$
%$$Q_{\infty}(y) = -\frac{12}{11}(y-\frac{1}{2})(y-1)^2(y-\frac{1}{6}) + y^4-1
C_{\infty}(x-\alpha_1)((x-1+\alpha_2)^2+\alpha_3)(x-\alpha_4)\label{polynomial naked singularity}
\end{equation}

\vspace{-5pt}

\noindent with $C_{\infty} = (-1+\alpha_1\alpha_2^2\alpha_4+\alpha_1\alpha_4-2\alpha_1\alpha_2\alpha_4+\alpha_1\alpha_3\alpha_4)^{-1}$ and $Q_{\infty}(y)  =  P_{\infty}(y)+y^4-1$.
These metrics satisfy the Einstein Condition \eqref{Einstein condition} with $\Lambda = -3$ for all $\alpha_1,\alpha_2,\alpha_3, \alpha_4\in\mathbb{R}$.

\subsection{A naked singularity in conformal infinity only}

\noindent  Setting $\alpha_2 = \alpha_3 = 0$ and choosing distinct $\alpha_1,\alpha_4 \in \mathbb{R}$ in \eqref{polynomial naked singularity}, the polynomial $P_{\infty}$ has a double root only at $x = 1$ while $Q_\infty$ has a simple root at $1$.
%the conformal infinity $y = x$ which is reached at the point $(x,y)=(1,1)$\todo{and a simple root for $Q$}. 
This will correspond to a \textit{Naked Singularity} in the metric, which we describe in Section \ref{naked singularity} once we ensure that our metric is smooth and Riemannian. We first need to verify that we have the correct signs $P_{\infty}>0$ and $Q_{\infty}<0$ on the region $\alpha_1 \leqslant x < y \leqslant1$. Assume that $\alpha_1 < 1 \leqslant\alpha_4$, then the inequality $P_{\infty}'(\alpha_1) = \frac{(\alpha_1-1)^2(\alpha_1-\alpha_4)}{\alpha_1 \alpha_4-1 } > 0$, which is satisfied whenever $\alpha_1 < \alpha_4^{-1}$, guarantees that $P_{\infty} >0$ on $(\alpha_1,1)$. To guarantee that $Q_\infty$ has the right sign, it is enough to impose $-1<
\alpha_1 < 0$. 
\begin{rem}
    This is true on a larger range of values of $\alpha_1$ which we do not attempt to describe. For instance, the graphs below show examples where the sign conditions for \ref{family PD with a} to be Riemannian are satisfied for $\alpha_1>0$, but the example $\alpha_1 = 0.01$, $\alpha_2=\alpha_3=0$ and $\alpha_4 = 7$ does not yield the right sign.
\end{rem}

Lastly, we assume that $\varphi$ and $\psi$ satisfy the periodicity conditions imposed in Proposition \ref{regularity criterion} to ensure that we find smooth metrics.

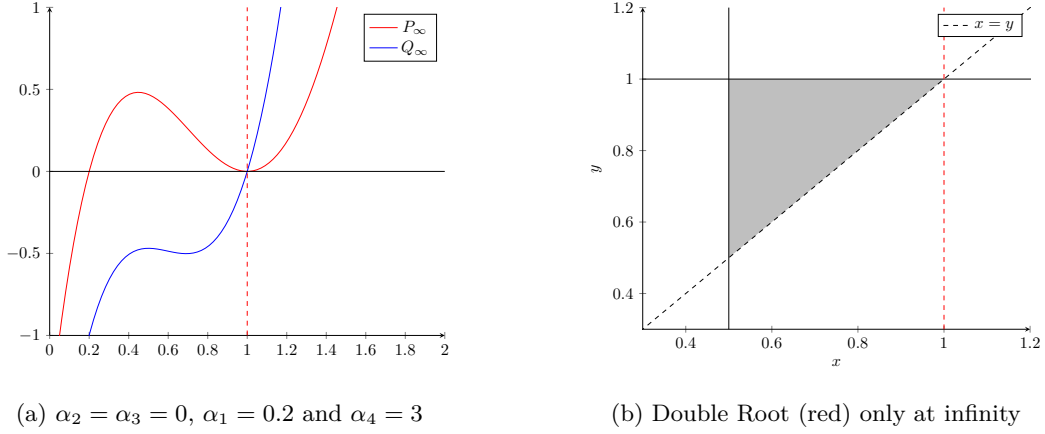
\begin{figure}[H]
\centering
\begin{subfigure}[t]{0.45\textwidth}
\begin{center}
\resizebox{0.8\linewidth}{!}{
\begin{tikzpicture}

\begin{axis}[
        axis lines = left,
        xlabel = \(\),
        ylabel = {\(\)},
        ymin = -1,
        ymax = 1,
        xmin = 0,
        xmax = 2,
    ]
    %Below the red parabola is defined
    \addplot [
        domain=0:2, 
        samples=100, 
        color=red,
    ]
    {-2.5*(x-0.2)*(x-1)^2*(x-3)};
    \addlegendentry{\(P_{\infty}\)}
    %Here the blue parabola is defined
    \addplot [
        domain=0:2, 
        samples=100, 
        color=blue,
        ]
        {-2.5*(x-0.2)*(x-1)^2*(x-3)+x^4-1};
    \addlegendentry{\(Q_{\infty}\)}
    
    \draw[red, dashed] (1, -1) -- (1,1);
    
    \addplot [
        domain=0:2, 
        samples=100, 
        color=black
        ]
        {0};
    \end{axis}
    
\end{tikzpicture}}
\end{center}

\vspace{-5pt}

\caption{$\alpha_2 = \alpha_3 = 0$, $\alpha_1 = 0.2$ and $\alpha_4 = 3$ }
\end{subfigure}
\begin{subfigure}[t]{0.45\textwidth}
\begin{center}
\resizebox{0.8\linewidth}{!}{
\begin{tikzpicture}

\begin{axis}[
        axis lines = left,
        xlabel = \(x\),
        ylabel = {\(y\)},
        ymin = 0.3,
        ymax = 1.2,
        xmin = 0.3,
        xmax = 1.2,
    ]
    
    \addplot [
        name path=g,
        domain=0:2, 
        samples=100, 
        color=black,
        dashed
    ]
    {x};
    \addlegendentry{\(x=y\)}
    \addplot [
        name path =f,
        domain=0:2, 
        samples=100, 
        color=black,
    ]
    {1};
    \draw (0.5,0 |- current axis.south) -- (0.5,2 |- current axis.north);
    \draw (1.5, 0 |- current axis.south) -- (1.5,2 |- current axis.north);
    \draw[red, dashed] (1, 0 |- current axis.south) -- (1,2 |- current axis.north);
    
    \addplot [
            thick,
            color=gray,
            fill=gray, 
            fill opacity=0.5
        ]
        fill between[
            of=f and g,
            soft clip ={domain=0.5:1},
        ];
    \end{axis}

\end{tikzpicture}
}
\end{center}

\vspace{-5pt}

\caption{Double Root (red) only at infinity}
\end{subfigure}

\vspace{-5pt}

\caption{Example polynomials and region for metric with naked singularity at infinity.} \label{Example}
\end{figure}
The metrics obtained in this way can be approached by perturbing the parameters $\alpha_2$ and $\alpha_3$ in various ways around $(0,0)$. This gives the following different types of degenerations degenerations, which we describe below.

 \paragraph{Degeneration 1: from a smooth metric to a naked singularity.}
 
 By taking $\alpha_3 > 0$ and keeping $\alpha_2 = 0$, the double root of $P_{\infty}$ at $1$ is replaced with two complex roots, see Figure \ref{figure complex conjugate boundary}. Taking the limit $\alpha_3\to 0$ yields the above naked singularity. Similarly, by taking $\alpha_2 < 0$ and $\alpha_3 = 0$, the double root of $P_{\infty}$ is moved past the conformal infinity $y=x$, see Figure \ref{double root after 1}. Taking the limit $\alpha_2\to 0$ yields the above naked singularity.
 
 Both of these situations yield a smooth metric at conformal infinity by Proposition \ref{regular conformal infinity}. Indeed, $P_\infty$ does not have any root close to the root of $Q_{\infty}$.
 
 \begin{figure}[H]
\centering
\begin{subfigure}[t]{0.45\textwidth}
\begin{center}
\resizebox{0.8\linewidth}{!}{
\begin{tikzpicture}

\begin{axis}[
    axis lines = left,
    xlabel = \(\),
    ylabel = {\(\)},
    ymin = -1,
    ymax = 1,
    xmin = 0,
    xmax = 2,
]
%Below the red parabola is defined
\addplot [
    domain=0:2, 
    samples=100, 
    color=red,
]
{-2.53807*(x-0.2)*((x-1)^2+0.01)*(x-3)};
\addlegendentry{\(P_{\infty}\)}
%Here the blue parabola is defined
\addplot [
    domain=0:2, 
    samples=100, 
    color=blue,
    ]
    {-2.53807*(x-0.2)*((x-1)^2+0.01)*(x-3)+x^4-1};
\addlegendentry{\(Q_{\infty}\)}

\addplot [
    domain=0:2, 
    samples=100, 
    color=black
    ]
    {0};

\end{axis}
\end{tikzpicture}}
\end{center}

\vspace{-10pt}

\caption{$\alpha_3 > 0$ and $\alpha_2 = 0$.}\label{figure complex conjugate boundary}
\end{subfigure}
\begin{subfigure}[t]{0.45\textwidth}
\begin{center}
\resizebox{0.8\linewidth}{!}{
\begin{tikzpicture}

\begin{axis}[
    axis lines = left,
    xlabel = \(\),
    ylabel = {\(\)},
    ymin = -1,
    ymax = 1,
    xmin = 0,
    xmax = 2,
]
%Below the red parabola is defined
\addplot [
    domain=0:2, 
    samples=100, 
    color=red,
]
{-2.95421*(x-0.2)*((x-1-0.05)^2)*(x-3)};
\addlegendentry{\(P_{\infty}\)}
%Here the blue parabola is defined
\addplot [
    domain=0:2, 
    samples=100, 
    color=blue,
    ]
    {-2.95421*(x-0.2)*((x-1-0.05)^2)*(x-3)+x^4-1};
\addlegendentry{\(Q_{\infty}\)}

\addplot [
    domain=0:2, 
    samples=100, 
    color=black
    ]
    {0};

\end{axis}
\end{tikzpicture}
}
\end{center}

\vspace{-10pt}

\caption{$\alpha_2 < 0$ and $\alpha_3 = 0$.}\label{double root after 1}
\end{subfigure}

\vspace{-5pt}

\caption{Smooth Metric to Naked Singularity} \label{smooth to naked}
\end{figure}
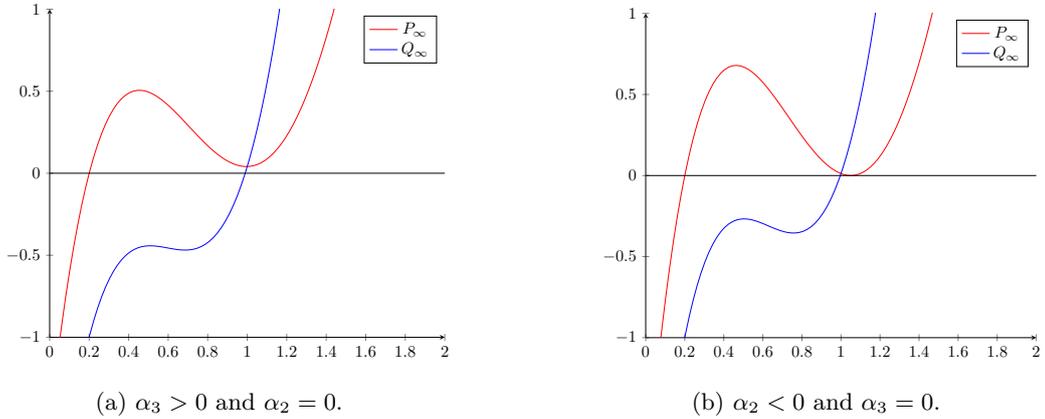

%By taking $\alpha_3 > 0$ and keeping $\alpha_2 = 0$, the double root of $P_{\infty}$ at $1$ is replaced with two complex roots, which yields a smooth metric.\todo{ref appendix} 

\paragraph{Degeneration 2: from a conical singularity to a naked singularity.}
By taking $\alpha_3 < 0$ and $\alpha_2 = 0$, see Figure \ref{figure angle topology}, the double root of $P_{\infty}$ at $1$ is split in two real roots $x_-<1<x_+$. This changes the topology and creates a codimension $2$ cone-edge singularity along $\{x=x_-\}$ by Lemma \ref{regularity at one rod}, extending to the conformal infinity $\{x=y\}$. As $\alpha_3\to 0$, the angle tends to zero and a \textit{naked singularity} appears while the singularities in the bulk are ``sent to infinity''. 

By setting $\alpha_3 < 0$ and $\alpha_2 = -\sqrt{-\alpha_3}$ as in Figure \ref{conical to naked b} %, we have 
%$(x-1+\alpha_2)^2+\alpha_3 = (x-1)(x-1-2\sqrt{-\alpha_3})$, thus
the double root of $P_{\infty}$ at $1$ is split into a single root at $1$ and a root larger than $1$. This gives a conical singularity in the metric at the conformal infinity \textit{only} this time. As $\alpha_3\to 0$, the angle tends to zero and a naked singularity appears in the limit.
 \begin{figure}[H]
\centering
\begin{subfigure}[t]{0.45\textwidth}
\begin{center}
\resizebox{0.8\linewidth}{!}{
\begin{tikzpicture}

\begin{axis}[
    axis lines = left,
    xlabel = \(\),
    ylabel = {\(\)},
    ymin = -1,
    ymax = 1,
    xmin = 0,
    xmax = 2,
]
%Below the red parabola is defined
\addplot [
    domain=0:2, 
    samples=100, 
    color=red,
]
{-2.46305*(x-0.2)*((x-1)^2-0.01)*(x-3)};
\addlegendentry{\(P_{\infty}\)}
%Here the blue parabola is defined
\addplot [
    domain=0:2, 
    samples=100, 
    color=blue,
    ]
    {-2.46305*(x-0.2)*((x-1)^2-0.01)*(x-3)+x^4-1};
\addlegendentry{\(Q_{\infty}\)}

\addplot [
    domain=0:2, 
    samples=100, 
    color=black
    ]
    {0};

\end{axis}
\end{tikzpicture}}
\end{center}

\vspace{-10pt}

\caption{$\alpha_3 < 0$ and $\alpha_2 = 0$. }\label{figure angle topology}
\end{subfigure}
\begin{subfigure}[t]{0.45\textwidth}
\begin{center}
\resizebox{0.8\linewidth}{!}{
\begin{tikzpicture}

\begin{axis}[
    axis lines = left,
    xlabel = \(\),
    ylabel = {\(\)},
    ymin = -1,
    ymax = 1,
    xmin = 0,
    xmax = 1.7,
]
%Below the red parabola is defined
\addplot [
    domain=0:2, 
    samples=100, 
    color=red,
]
{-3.1731*(x-0.2)*((x-1-0.0707)^2-0.005)*(x-3)};
\addlegendentry{\(P_{\infty}\)}
%Here the blue parabola is defined
\addplot [
    domain=0:2, 
    samples=100, 
    color=blue,
    ]
    {-3.1731*(x-0.2)*((x-1-0.0707)^2-0.005)*(x-3)+x^4-1};
\addlegendentry{\(Q_{\infty}\)}

\addplot [
    domain=0:2, 
    samples=100, 
    color=black
    ]
    {0};

\end{axis}
\end{tikzpicture}
}
\end{center}

\vspace{-10pt}

\caption{$\alpha_3 < 0$ and $\alpha_2 = -\sqrt{-\alpha_3}$. }\label{conical to naked b}
\end{subfigure}

\vspace{-5pt}

\caption{Conical Singularity to Naked Singularity} \label{conical to Naked}
\end{figure}
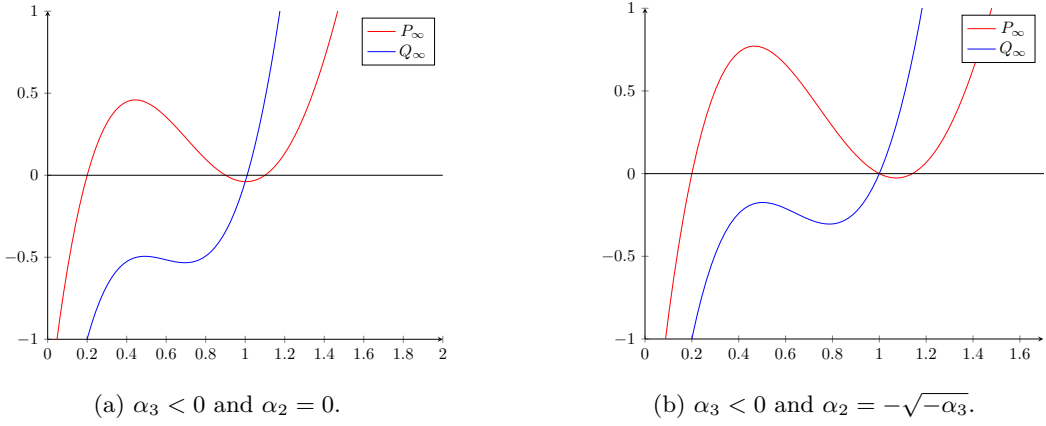

%\todo[inline]{Is it always true that $ Q_\infty $ has a double root in the case (b)? If so, we should note it somewhere, if not, we should choose another example not to distract the reader with another degenerate situation} Its not true

%By taking $\alpha_3 < 0$ and $\alpha_2 = 0$, this replaces the double root at $1$ with two real roots, which yields conical singularities in the metric.\todo{ref appendix}

%By setting $\alpha_3 < 0$ and $\alpha_2 = -\sqrt{-\alpha_3}$, we have 
%$(x-1+\alpha_2)^2+\alpha_3 = (x-1)(x-1-2\sqrt{-\alpha_3})$, thus the double root of $P_{\infty}$ at $1$ is split into a single root at $1$ and a root larger than $1$. This gives a conical singularity in the metric at the conformal infinity only.\todo{ref appendix}

\vspace{-13pt}

\hspace{-12pt}\begin{minipage}{0.63\textwidth}
\paragraph{Degeneration 3: from cusp to naked singularity.} By taking $\alpha_2 > 0$ and $\alpha_3 = 0$, the double root $1-\alpha_2$ of $P_{\infty}$ is moved to the left of the root in $Q_\infty$. This creates a cusp in the bulk metric as well as in its infinity by Sections \ref{section appendix possible asymptotics} and \ref{section appendix possible asymptotics boundary} in the appendix. 
\\

As in Section \ref{section AdS C metric}, this cusp at $\{x=1-\alpha_2\}$ separates the manifold in two regions infinitely far apart, and the conformal infinity in two finite volume manifolds with cusp ends. 
\\

When $\alpha_2\to 0$, the volume of $\{1-\alpha_2<x=y<1\}$ tends to zero and the region disappears, the metric on $\{\alpha_1<x=y<1-\alpha_2\}$ has infinite diameter for $\alpha_2>0$ but finite diameter in the limit $\alpha_2\to 0$ (these remarks do not depend on the representative of the conformal class).
\end{minipage}
\begin{minipage}{0.35\textwidth}
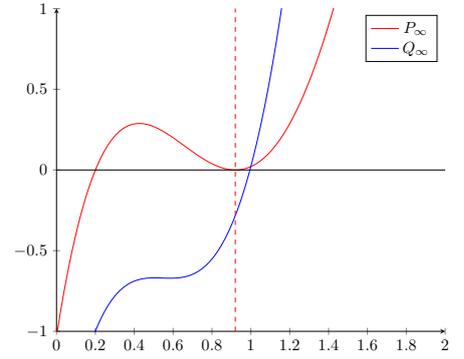
\begin{figure}[H]
\begin{center}

\resizebox{\linewidth}{!}{\begin{tikzpicture}

\begin{axis}[
    axis lines = left,
    xlabel = \(\),
    ylabel = {\(\)},
    ymin = -1,
    ymax = 1,
    xmin = 0,
    xmax = 2,
]
%Below the red parabola is defined
\addplot [
    domain=0:2, 
    samples=100, 
    color=red,
]
{-2.03186*(x-0.2)*((x-1+0.08)^2)*(x-3)};
\addlegendentry{\(P_{\infty}\)}
%Here the blue parabola is defined
\addplot [
    domain=0:2, 
    samples=100, 
    color=blue,
    ]
    {-2.03186*(x-0.2)*((x-1+0.08)^2)*(x-3)+x^4-1};
\addlegendentry{\(Q_{\infty}\)}

\draw[red, dashed] (1-0.08, -1) -- (1-0.08,1);

\addplot [
    domain=0:2, 
    samples=100, 
    color=black
    ]
    {0};

\end{axis}
\end{tikzpicture}}

\vspace{-10pt}

\caption{Cusp to Naked Singularity}
\end{center}
\end{figure}
\end{minipage}

%We can approach this metric by choosing parameters $\alpha_1,\alpha_2,\alpha_3,\alpha_4$ and setting 

%$$P(x) = \frac{1}{-1+\alpha_1\alpha_2^2\alpha_4+\alpha_1\alpha_4-2\alpha_1\alpha_2\alpha_4+\alpha_1\alpha_3\alpha_4}(x-\alpha_1)((x-1+\alpha_2)^2+\alpha_3)(x-\alpha_4)$$

%$Q(y) = \frac{1}{-1+\alpha_1\alpha_2^2\alpha_4+\alpha_1\alpha_4-2\alpha_1\alpha_2\alpha_4+\alpha_1\alpha_3\alpha_4}(y-\alpha_1)((y-1+\alpha_2)^2+\alpha_3)(y-\alpha_4)+y^4-1$$
%which satisfies the Einstein Conditions and the desired signs. 

\vspace{-3pt}

This is a manifestation of cusp degenerations in the bulk manifold comparable to those of Section \ref{section AdS C metric}. Indeed, in the family of metrics obtained from \eqref{polynomial naked singularity}, there is a $4$-dimensional family of smooth Poincaré-Einstein metrics with a ($3$-dimensional) boundary constituted of metrics with one cusp separating the manifold in two set, and a $2$-dimensional family with two cusps separating the manifold in three.

\vspace{-2pt}

\begin{rem}
    There are important differences with Section \ref{section AdS C metric}. The cusps from \eqref{family PD with a} for $a=1$ are ``twisted'' (see \eqref{separating cusp}) and do not look like mere products of surfaces in the limit. Moreover, as described above, as $\alpha_2\to 0$ the cusps ``escapes'' to infinity creating the above unexpected naked singularity at infinity. This was impossible in the family \eqref{C-metric} because the double root in $P$ could not approach $0$ and the double root in $Q$ could not approach $-1$.
\end{rem}

\subsection{Two cusps at conformal infinity only}

\vspace{-13pt}

\hspace{-12pt}\begin{minipage}{0.56\textwidth}
This time, we exhibit a metric with codimension $2$ cusps ends at the conformal infinity \textit{only} -- in particular, the conformal infinity is not compact. Unlike the example of Section \ref{section degeneration ads c metric} these cusps do not cut the manifold in different pieces. Consider 

\vspace{-5pt}

$$
\begin{aligned}
P_{2,\infty}(x) &= -\frac{1}{2}(x-1)^3(x+1) \text{ and}\\ Q_{2,\infty}(y) &= \frac{1}{2}(y-1)(y+1)^3
\end{aligned}$$

\vspace{-5pt}

which are limit of the polynomials in \eqref{polynomial naked singularity} for $\alpha_1=-1$, $\alpha_2=0$, $\alpha_3=0$ and $\alpha_4=1$. These polynomials have the desired signs on the region $-1 \leqslant x \leqslant 1, -1 \leqslant y \leqslant1$ making the metric \eqref{family PD with a} with $a=1$ Riemannian. Its infinity has \textit{two} cusp ends at the points $(-1,-1)$ and $(1,1)$ thanks to \eqref{dvp triple root simple root}. This is a limiting case for all the previous degenerations as well as a limit of naked singularities at either $1$ or $-1$.

\end{minipage}\hfill
\begin{minipage}{0.42\textwidth}
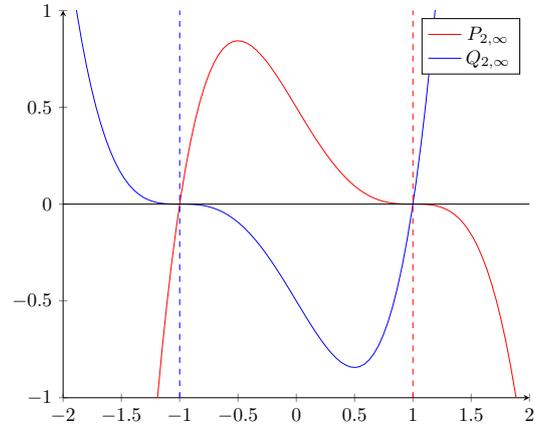
\begin{figure}[H]
\begin{center}
\resizebox{\linewidth}{!}{\begin{tikzpicture}[scale=0.8]

\begin{axis}[
    axis lines = left,
    xlabel = \(\),
    ylabel = {\(\)},
    ymin = -1,
    ymax = 1,
    xmin = -2,
    xmax = 2,
]
%Below the red parabola is defined
\addplot [
    domain=-2:2, 
    samples=100, 
    color=red,
]
{-0.5*(x-1)^3*(x+1)};
\addlegendentry{\(P_{2,\infty}\)}
%Here the blue parabola is defined
\addplot [
    domain=-2:2, 
    samples=100, 
    color=blue,
    ]
    {0.5*(x-1)*(x+1)^3};
\addlegendentry{\(Q_{2,\infty}\)}

\addplot [
    domain=-2:2, 
    samples=100, 
    color=black
    ]
    {0};

\draw[red, dashed] (1, -1) -- (1,1);

\draw[blue, dashed] (-1, -1) -- (-1,1);

\end{axis}
\end{tikzpicture}}

\vspace{-10pt}

\caption{Two cusps at conformal infinity \textit{only}.}
\end{center}
\end{figure}
\end{minipage}
\begin{rem}
    From Section \ref{section PD}, one moreover notices that this metric is \emph{anti-selfdual} since the linear and cubic coefficients of $P_{\infty,2}$ are opposite.
\end{rem}

%\noindent \textbf{Sample surface of revolution plot as a test, feel free to remove and play around with}

%\begin{center}
%\begin{tikzpicture}
%\begin{axis}[grid=major]
%\addplot3 [
%surf,
%shader=faceted,
%samples=25,
%domain=0:2,
%y domain=0:1,
%] {(x^2+y^2)^{1/3}};
%\end{axis}
%\end{tikzpicture}

%\end{center}
\appendix

\section{Possible asymptotics of the bulk metrics and regularity}\label{section appendix possible asymptotics}
Let us describe here the possible asymptotics of our metrics and give the regularity conditions. The regularity conditions obtained for toric metrics are classical and we focus on ruling our conical singularities.

\subsection{At a simple root $x_1$ of $P$ and a generic point $y$}\label{one simple root}

As defined for instance in \cite{MR3065256}, a metric with cone-edge singularity of angle $2\pi\beta>0$ along a codimension $2$ submanifold $\Sigma$ has the following asymptotic at $\Sigma$: for a $2\pi$-periodic $\theta$ and a $1$-form $\omega$ on $\Sigma$
\begin{equation}
    dr^2 + \beta^2 r^2 (d\theta + \omega)^2 + g_\Sigma + \mathcal{O}(r^{1+\epsilon}), \text{ for $\epsilon>0$}.\label{conical al}
\end{equation}

\begin{lem}\label{regularity at one rod}
    At a simple root $x_1$ of $P$ and a generic point $y\notin\{-1/x_1,1/x_1\}$, our metric \eqref{family PD with a} with $a\in\{0,1\}$ has a cone-edge singularity whose angle is the period of $\frac{|P'(x_1)|}{2(1+a^2x_1^4)}\theta_1$, where $\theta_1:= \varphi + ax_1^2\psi$.
\end{lem}
\begin{proof}
To do this, we first note that close to its root $x_1$, we have $P(x) = P'(x_1)(x-x_1) + O((x-x_1)^2)$ at first order. Close to the root $x=x_1$ and at $y\neq\pm x_1^{-1}$ which is not a root of $Q$, the metric therefore reads:
\begin{align*}
    g_{PD} = \frac{1}{(x_1-y)^2}&\left[-\frac{Q(y)}{1-a^2x_1^2y^2}(d\psi-ax_1^2d\varphi)^2 - \frac{1-a^2x_1^2y^2}{Q(y)}dy^2\right.\\
    &\left. + \frac{P'(x_1)(x-x_1)}{1-a^2x_1^2y^2}(d\varphi-ay^2d\psi)^2+\frac{1-a^2x_1^2y^2}{P'(x_1)(x-x_1)}dx^2\right] + \mathcal{O}((x-x_1)^2).
\end{align*}
Our codimension $2$ submanifold $\Sigma$ is given by $\{x=x_1\}$, hence $dx=0$ and $\theta_1 := \varphi + ax_1^2\psi =cst$, that is $d\theta_1=d\varphi + ax_1^2d\psi=0$ (this is chosen as the orthogonal of the $1$-form $d\psi-ax_1^2d\varphi$). The local coframe on $\Sigma$ we will use is therefore $ dy $ and $\omega_1:=d\psi-ax_1^2d\varphi$. With these notations, the metric becomes:
\begin{align*}
    g_{PD} = \frac{1}{(x_1-y)^2}&\left[-\frac{Q(y)}{1-a^2x_1^2y^2}\omega_1^2 - \frac{1-a^2x_1^2y^2}{Q(y)}dy^2\right.\\
    &\left. + (1-a^2x_1^2y^2)\left( \frac{P'(x_1)(x-x_1)}{(1+a^2x_1^4)^2}\left(d\theta_1 + f(y,x_1)\omega_1\right)^2+\frac{dx^2}{P'(x_1)(x-x_1)}\right)\right]+ \mathcal{O}((x-x_1)^2).
\end{align*}
where $f(y,x_1)=\frac{x_1^2+y^2}{1-x_1^2y^2}$ if $a= 1$ and $f(y,x_1)=0$ if $a=0$. The regularity of the metric close to $x=x_1$ therefore reduces to the regularity of $\frac{P'(x_1)(x-x_1)}{(1+x_1^4)^2}\left(d\theta_1 + f(y,x_1)\omega_1\right)^2+\frac{1}{P'(x_1)(x-x_1)}dx^2$.

Considering a change of variables $x=x_1+\frac{P'(x_1)}{4}r^2$, we find the conical singularity metric:
$$\frac{|P'(x_1)|^2}{4(1+a^2x_1^4)^2}r^2\left(d\theta_1 + f(y,x_1)\omega_1\right)^2+dr^2,$$ of angle the period of $\frac{|P'(x_1)|}{2(1+a^2x_1^4)}\theta_1$ by comparison with \eqref{conical al}. It will be smooth if and only if this period is $2\pi$.
\end{proof}

The case of a simple root $y_2$ of $Q$ is treated similarly and yields a cone-edge singularity whose angle is given by the period of $\frac{|Q'(y_2)|}{2(1+a^2y_2^4)}\theta_2$, where $\theta_2$ is $\psi + y_2^2\varphi$.
\subsection{At $x_1$ simple root of $P$ and $y_2$ simple root of $Q$}
Let us now consider $y = y_2$, a simple root of $Q$, and still assume that $x_1^2\neq y_2^{-2}$ if $a=1$. 

\begin{lem}\label{regularity at intersection of rods}
    Assume that $x_1$ is a simple root of $P$, and $y_2$ is a simple root of $Q$, and $1-a^2x_1^2y_2^2\neq 0$.  Then the metric \eqref{family PD with a} is smooth at $(x_1,y_2)$ if and only if both $\frac{|P'(x_1)|}{2(1+a^2x_1^4)}\theta_1$ and $\frac{|Q'(y_2)|}{2(1+a^2y_2^4)}\theta_2$ are $2\pi$-periodic, where $\omega_1=d\psi-ax_1^2d\varphi$ and $\omega_2=d\varphi-ay_2^2d\psi$.
\end{lem}
\begin{proof}
Expanding the metric near $p=x_1$ and $q=y_2$, a first order development of the metric gives:
\begin{align*}
    g_{PD} = \frac{1}{(x_1-y_2)^2}&\left[-\frac{Q'(y_2)(y-y_2)}{1-a^2x_1^2y_2^2}(d\theta_2 + \tilde{\omega}_2)^2 - \frac{1-a^2x_1^2y_2^2}{Q'(y_2)(y-y_2)}dy^2\right.\\
    &\left. + (1-a^2x_1^2y_2^2)\left( \frac{P'(x_1)(x-x_1)}{(1+a^2x_1^4)^2}\left(d\theta_1 + \tilde{\omega}_1\right)^2+\frac{1}{P'(x_1)(x-x_1)}dx^2\right)\right] \\
    &+ \mathcal{O}((x-x_1)^2 + (y-y_2)^2).
\end{align*}
for some $1$-forms $\tilde{\omega}_1 = f_1(y,x_1)\omega_1$ and $\tilde{\omega}_2 = f_2(x,y_2)\omega_2$ for some smooth functions $f_1$ and $f_2$ whose explicit value does not affect the regularity ($f_1=f_2=0$ if $a=0$), and where $\omega_1=d\psi-ax_1^2d\varphi$ and $\omega_2=d\varphi-ay_2^2d\psi$.

The same change of variables as in Section \ref{one simple root} in both $x$ and $y$ ensures that the metric is smooth at $(x_1,y_2)$ if and only if $\frac{|P'(x_1)|}{2(1+a^2x_1^4)}\theta_1$ and $\frac{|Q'(y_2)|}{2(1+a^2y_2^4)}\theta_2$ are $2\pi$-periodic.
\end{proof}

We conclude with the following regularity proposition.
\begin{prop}\label{regularity criterion}
    Let $P$ and $Q$ be polynomials such that $P>0$ and $Q<0$ on $(x_1,y_2)\subset\mathbb{R}$ and assume that:  $x_1$ is a simple root of $P$, $y_2$ is a simple root of $Q$, and $1-a^2x^2y^2\neq 0$ for $x,y\in [x_1,y_2]$.
    
    Then, the metric \eqref{family PD with a} is smooth if and only if the variables $\frac{|P'(x_1)|}{2(1+a^2x_1^4)}\theta_1$ and $\frac{|Q'(y_2)|}{2(1+a^2y_2^4)}\theta_2$ are $2\pi$-periodic where $\theta_1:= \varphi + ax_1^2\psi$ and $\theta_2:= \psi + ay_2^2 \varphi$.
\end{prop}

\subsection{At a double root $x_1$ of $P$ and a generic $y$: separating cusp}\label{section one double root}

Similarly, close to $x_1$ a double root of the polynomial, one has $P(x) \approx P''(x_1)(x-x_1)^2/2$. As in Section \ref{one simple root}, we find that close to the same codimension $2$ submanifold $\Sigma$, the metric is asymptotic to:
\begin{equation}
    (1-a^2x_1^2y^2)\left(\frac{P''(x_1)(x-x_1)^2}{2(1+a^2x_1^4)^2}\left(d\theta_1 + f(y,x_1)\omega_1\right)^2+\frac{2dx^2}{P''(x_1)(x-x_1)^2}\right)+ g_\Sigma\label{separating cusp}
\end{equation}
for some smooth $y\mapsto f(y,x_1)$ vanishing when $a=0$ which is an asymptotically cuspidal metric. This is a smooth complete metric but it adds an end to the manifold. 

\subsection{Approaching a cuspidal end}\label{section approaching cusp}

Let us now explain how one can approach a codimension $2$ cuspidal end by smooth metrics. Assume that $x_1\pm i\epsilon$ are two complex conjugate roots of $P_\epsilon$ for $\epsilon>0$ that we will send to $0$. This time, we have the following second order approximation for $P_\epsilon(x)$ for $x$ close to $x_1$:
$P_\epsilon(x) \approx P''_\epsilon(x_1)\left((x-x_1)^2+\epsilon^2\right)/2 + \mathcal{O}((x-x_1)^3)$.

This implies that the metric is approximately
\begin{equation}
    (1-a^2x_1^2y^2)\left(\frac{P''_\epsilon(x_1)\left((x-x_1)^2+\epsilon^2\right)}{2(1+a^2x_1^4)^2}\left(d\theta_1 + f(y,x_1)\omega_1\right)^2+\frac{2dx^2}{P''_\epsilon(x_1)\left((x-x_1)^2+\epsilon^2\right)}\right)+ g_\Sigma.\label{separating cusp almost}
\end{equation}
This is a smooth metric, but along $\{x=x_1\}$ it is close to a thin cylinder, see the left picture of Figure \ref{figure cusp 3D}. As $\epsilon\to 0$, close to any $x\neq x_1$, the metric Cheeger-Gromov converges to the cuspidal metric \eqref{separating cusp} on compact sets, see the right image in Figure \ref{figure cusp 3D}.

\begin{figure}[H]
\centering 
\begin{subfigure}[t]{0.45\textwidth}
\includegraphics[scale = 0.15,trim={0 60pt 0 200pt},clip]{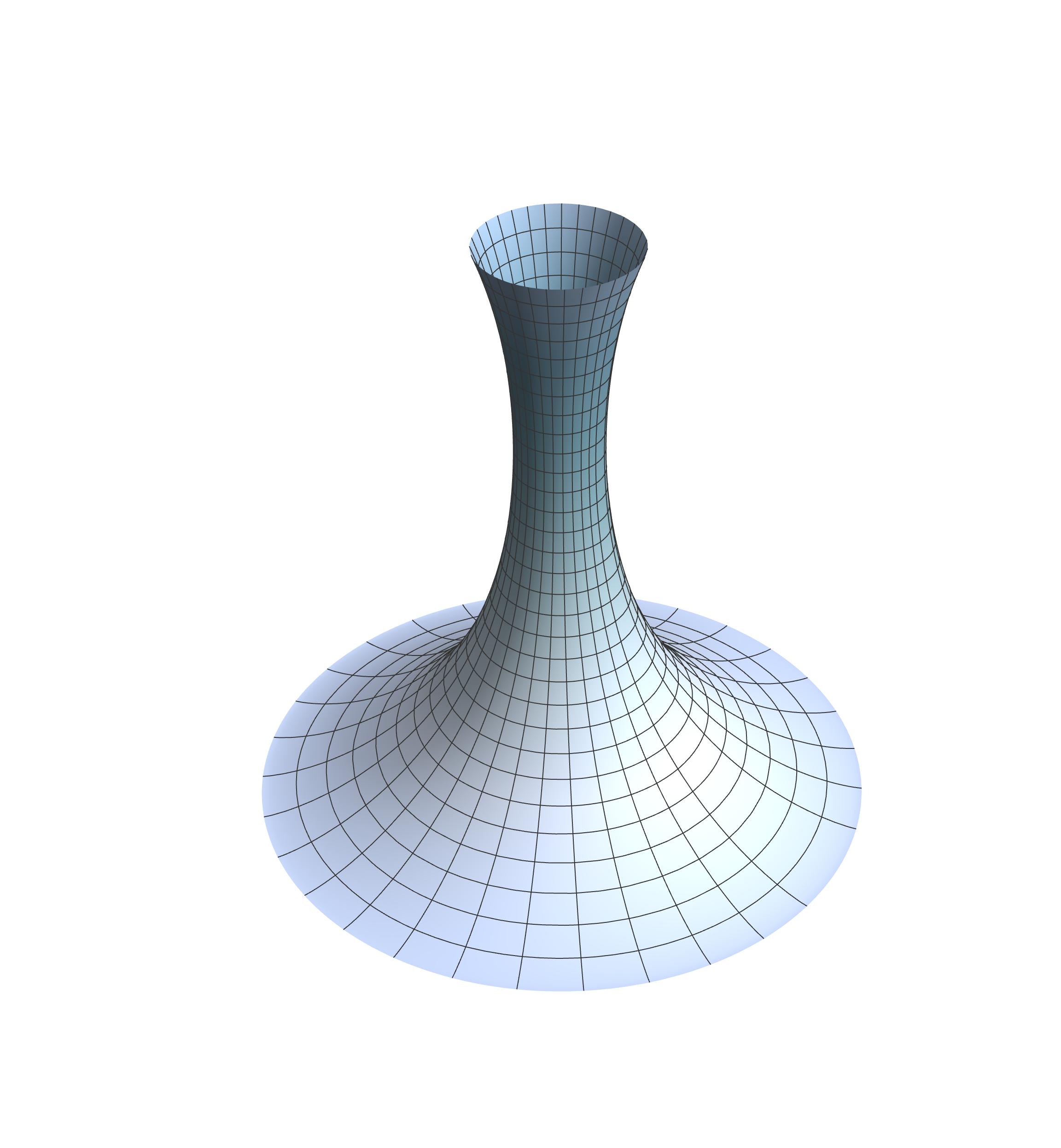}
\end{subfigure}
\begin{subfigure}[t]{0.45\textwidth}
\includegraphics[scale = 0.15,trim={0 60pt 0 200pt},clip]{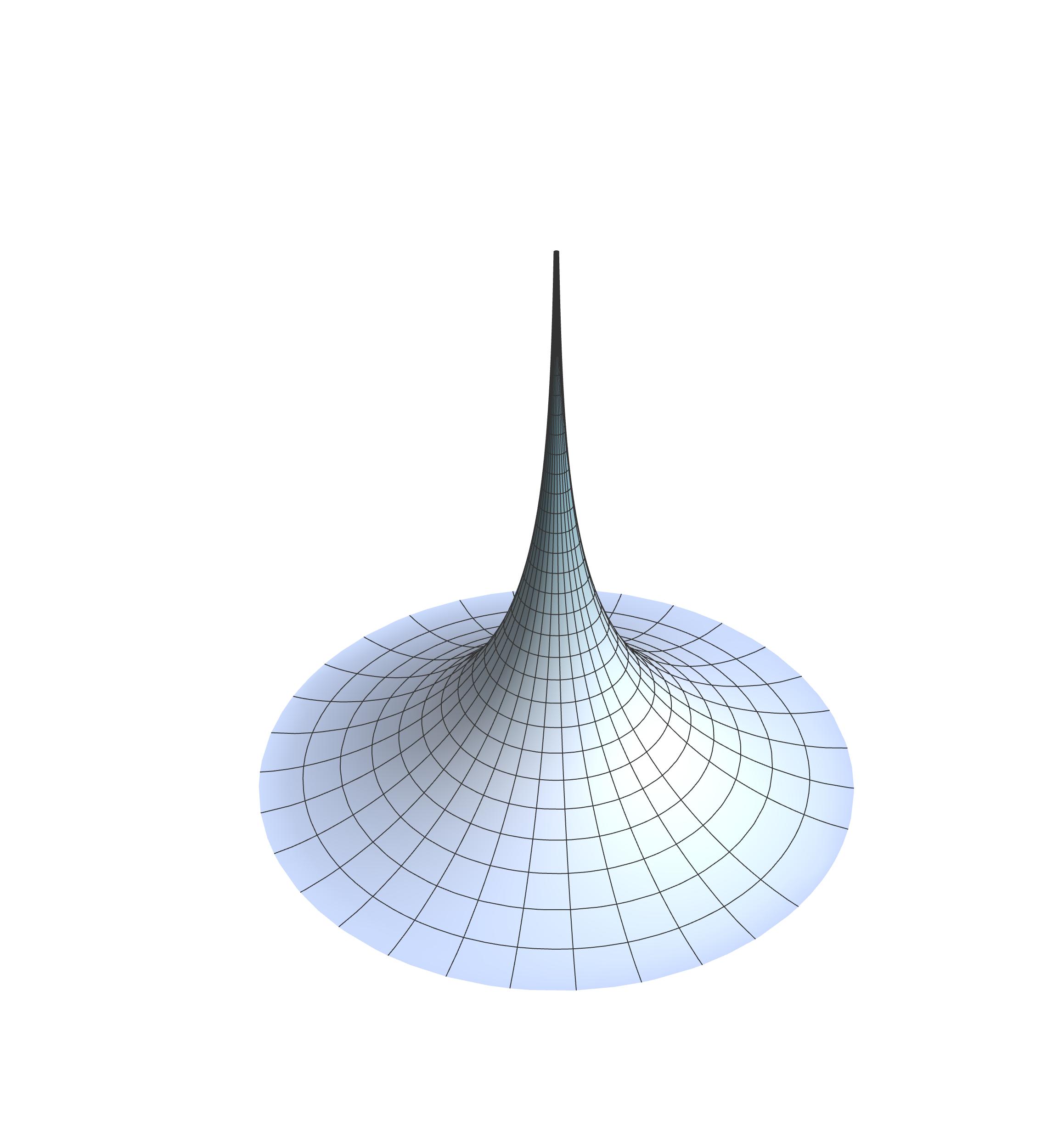}
\end{subfigure}

\vspace{-20pt}

\caption{Stages of cusp formation. The cusp on the right is infinitely long.}\label{figure cusp 3D}
\end{figure}

\section{Possible behaviors of the conformal boundary metrics}\label{section appendix possible asymptotics boundary}

As for the bulk metric, the conformal infinity $\{x=y\}$ has different possible asymptotic behaviors close to roots of $P$ or $Q$. We consider \eqref{family PD with a}, whose conformal metric at infinity is:
\begin{equation}
    g_{bdry}=(1-a^2x^4)\left(\frac{1}{P(x)}- \frac{1}{Q(x)}\right)dx^2-\frac{Q(x)}{1-a^2x^4}(d\psi-ax^2d\varphi)^2 + \frac{P(x)}{1-a^2x^4}(d\varphi-ax^2d\psi)^2.\label{conformal boundary metric}
\end{equation}
We will moreover assume that the regularity conditions for the bulk of Proposition \eqref{regularity criterion} are satisfied whenever applicable. Simpler arguments than Sections \ref{one simple root} and \ref{section one double root} imply the following result. 
\begin{prop}\label{regular conformal infinity}
    Under the assumptions of Proposition \ref{regularity criterion}, the conformal metric is smooth. Moreover, if $P$ (or $Q$) has a double root at $x_0$, then the conformal boundary metric of \eqref{conformal boundary metric} has a codimension $2$ separating cusp as described in \eqref{separating cusp}.
\end{prop}

We will now focus on the case $a=1$ of \eqref{family PD with a} and we will see that allowing the roots to be at $\pm1$ leads to different degenerate behavior for the conformal infinity alone.

\subsection{At $\pm 1$ a simple root of both $P$ and $Q$: conical singularity}

Let us assume that $1$ is a simple root of both $P$ and $Q$ for the metric \eqref{conformal boundary metric}. The case of $-1$ is treated similarly. As $ x\to 1 $, we obtain that the metric \eqref{conformal boundary metric} is asymptotic to
$$ C_1 dx^2 + C_2 \theta_1(x)^2 + C_3 (x-1)^2\theta_2(x)^2 $$
where $\theta_1(x) \to d\varphi-d\psi$, $\theta_2(x)\to d\varphi+d\psi$ as $x\to 1$, and $C_1 = 4 \left(\frac{1}{P'(1)} - \frac{1}{Q'(1)}\right)$, $C_2 =  \frac{P'(1) - Q'(1)}{4}$ and  $C_3 = -\frac{P'(1)Q'(1)}{P'(1) - Q'(1)} = \left(\frac{1}{P'(1)} - \frac{1}{Q'(1)}\right)^{-1}$.

This yields a codimension $2$ cone-edge singularity of angle $\frac{P'(1)Q'(1)}{2(P'(1) - Q'(1))}$ times the period of $\varphi+\psi$.

\subsection{At $\pm 1$ double root of $P$ and simple root of $Q$: naked singularity} \label{naked singularity}

\vspace{-15pt}

\hspace{-12pt}\begin{minipage}{0.60\textwidth}

Similarly, assuming that $1$ is a double root of $P$ and a simple root of $Q$, the metric \eqref{conformal boundary metric} approaches
\begin{equation}
    \frac{C_1}{(1-x)} dx^2 + C_2 \theta_1(x)^2 + C_3 (x-1)^3\theta_2(x)^2\label{dvp double root simple root}
\end{equation}
for $\theta_1(x) \to d\varphi-d\psi$, $\theta_2(x)\to d\varphi+d\psi$ as $x\to 1$ and where $C_1 = \frac{8}{P''(1)}$, $C_2 = -\frac{Q'(1)}{4}$ and $4C_2C_3 = -\frac{P''(1)}{2}Q'(1)$ so $C_3 = \frac{P''(1)}{2}$. A change of variables $r=2\sqrt{1-x}$ in \eqref{dvp double root simple root} yields the \textit{naked singularity} metric: close to $r=0$,
\begin{equation}
    C_1 dr^2 + C_2 \theta_1(x)^2 + \frac{C_3}{4} r^6\theta_2(x)^2.\label{dvp double root simple root change of var}
\end{equation}
\end{minipage}
\begin{minipage}{0.37\textwidth}
    
    \begin{figure}[H]
\includegraphics[scale = 0.15,trim={0 0 0 150pt},clip]{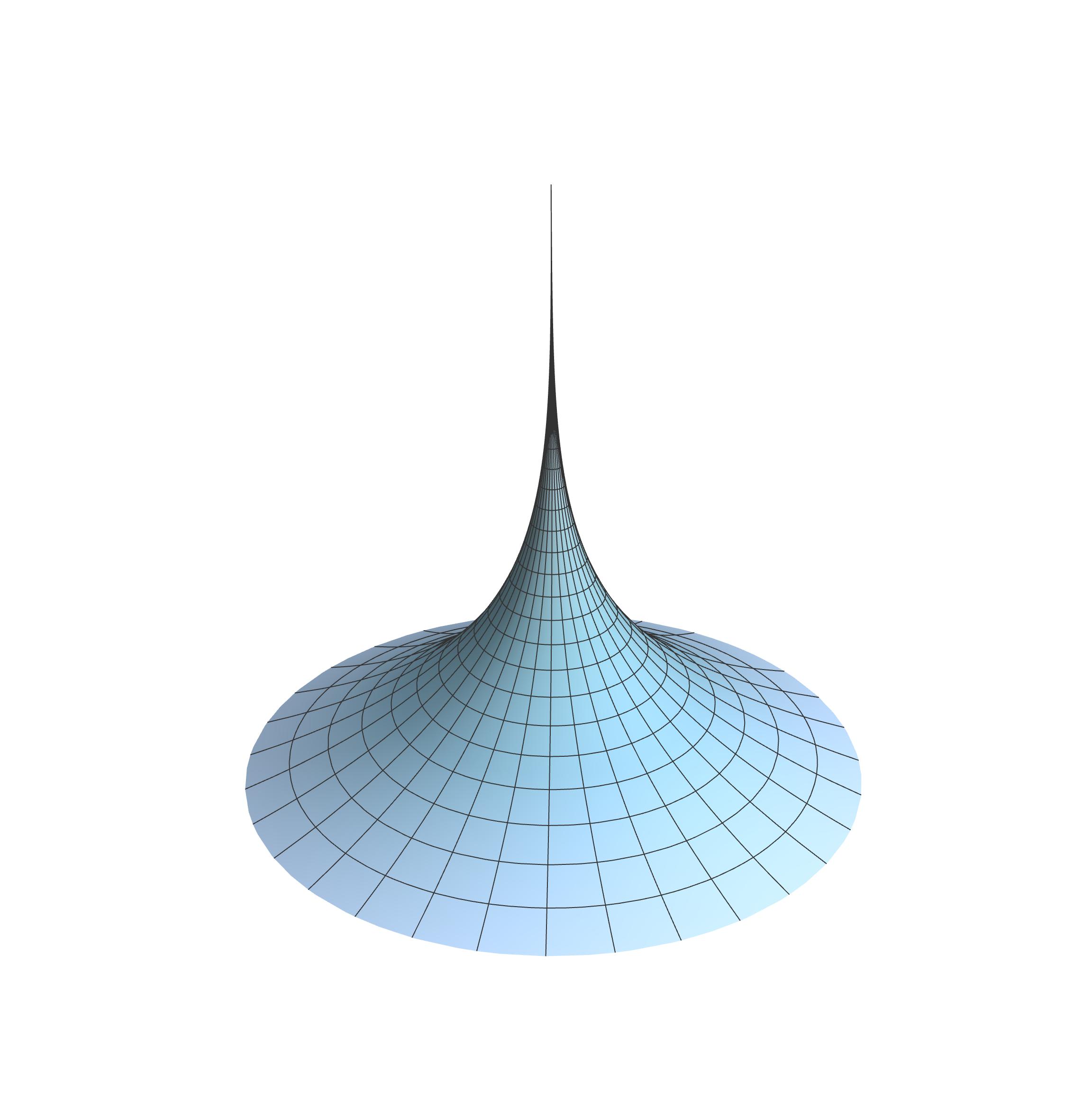}
%Will edit picture later

\vspace{-25pt}

\caption{Naked Singularity, \textbf{not} infinitely long.} \label{Naked Singularity Plot}
\end{figure}
\end{minipage}

\subsection{At $\pm1$ triple root of $P$ and simple root of $Q$: cusp end}

We finally assume that $1$ is a triple root of $P$ and a simple root of $Q$. The metric \eqref{conformal boundary metric} is asymptotic to
\begin{equation}
    \frac{C_1}{(1-x)^2} dx^2 + C_2 \theta_1(x)^2 + C_3 (x-1)^4\theta_2(x)^2\label{dvp triple root simple root}
\end{equation}
where again, $\theta_1(x) \to d\varphi-d\psi$, $\theta_2(x)\to d\varphi+d\psi$ as $x\to 1$ and where $C_1 = \frac{24}{P^{(3)}(1)}$, $C_2 = -\frac{Q'(1)}{4}$ and $C_3 = \frac{P^{(3)}(1)}{6}$. A change of variables $r=-\log(1-x)$ in \eqref{dvp triple root simple root} yields the \textit{cusp end} metric: for $r$ close to $+\infty$,
\begin{equation}
    C_1 dr^2 + C_2 \theta_1(x)^2 + C_3 e^{-4r}\theta_2(x)^2.\label{dvp double root simple root change of var boundary}
\end{equation}
\bibliographystyle{alpha}
\bibliography{ref.bib}
\nocite{*}

\Addresses

\end{document}